\documentclass[
a4paper 
,11pt 
,leqno 
,english 
]{scrartcl} 

\usepackage{my-preamble-article} 
\usepackage{my-macros} 
\usepackage{my-specific-macros} 
\usepackage{makecell} 
\title{Generic transversality of radially symmetric stationary solutions stable at infinity for parabolic gradient systems}
\author{Emmanuel \textsc{Risler}}
\begin{document}
\maketitle
\begin{abstract}
This paper is devoted to the generic transversality of radially symmetric stationary solutions of nonlinear parabolic systems of the form
\[
\partial_t w(x,t) = -\nabla V\bigl(w((x,t))\bigr) + \Delta_x w(x,t)
\,, 
\]
where the space variable $x$ is multidimensional and unbounded. It is proved that, generically with respect to the potential $V$, radially symmetric stationary solutions that are \emph{stable at infinity} (in other words, that approach a minimum point of $V$ at infinity in space) are transverse; as a consequence, the set of such solutions is discrete. This result can be viewed as the extension to higher space dimensions of the generic elementarity of symmetric standing pulses, proved in a companion paper. It justifies the generic character of the discreteness hypothesis concerning this set of stationary solutions, made in another companion paper devoted to the global behaviour of (time dependent) radially symmetric solutions stable at infinity for such systems. 
\end{abstract}
\nnfootnote{%
}
\pagebreak
\tableofcontents
\pagebreak
\section{Introduction}
\subsection{An insight into the main result}
The purpose of this paper is to prove the generic transversality of radially symmetric stationary solutions stable at infinity for gradient systems of the form
\begin{equation}
\label{parabolic_system_higher_space_dimension}
\partial_t w(x,t) = -\nabla V\bigl(w((x,t))\bigr) + \Delta_x w(x,t)
\,,
\end{equation}
where time variable $t$ is real, space variable $x$ lies in the spatial domain $\rr^{\dSpace}$ with $\dSpace$ an integer not smaller than $2$, the state function $(x,t)\mapsto w(x,t)$ takes its values in $\rr^{\dState}$ with $\dState$ a positive integer, and the nonlinearity is the gradient of a scalar potential function $V:\rr^{\dState}\to\rr$, which is assumed to be regular (of class at least $\ccc^2$). An insight into the main result of this paper (\vref{thm:main}) is provided by the following corollary. 
\begin{corollary}
\label{cor:insight_main_result}
For a generic potential $V$, the following conclusions hold:
\begin{enumerate}
\item every radially symmetric stationary solution stable at infinity of system \cref{parabolic_system_higher_space_dimension} is robust with respect to small perturbations of $V$;
\label{item:cor_insight_robustness}
\item the set of all such solutions is discrete. 
\label{item:cor_insight_discreteness}
\end{enumerate}
\end{corollary}
The discreteness stated in conclusion \cref{item:cor_insight_discreteness} of this corollary is a required assumption for the main result of \cite{Risler_globalBehaviourRadiallySymmetric_2017}, which describes the global behaviour of radially symmetric (time dependent) solutions stable at infinity for the parabolic system \cref{parabolic_system_higher_space_dimension}. \Cref{cor:insight_main_result} provides a rigorous proof that this assumption holds generically with respect to $V$. 

This paper can be viewed as a supplement of the article \cite{JolyRisler_genericTransversalityTravStandFrontsPulses_2023}, which is devoted to the generic transversality of bistable travelling fronts and standing pulses stable at infinity for parabolic systems of the form \cref{parabolic_system_higher_space_dimension} in (unbounded) space dimension one, and which provides a rigorous proof of the genericity of similar assumptions made in \cite{Risler_globalRelaxation_2016,Risler_globalBehaviour_2016,Risler_globalBehaviourHyperbolicGradient_2017}. The ideas, the nature of the results, and the scheme of the proof are the same.  
\subsection{Radially symmetric stationary solutions stable at infinity}
A function $u:[0,+\infty)\to\rr^{\dState}$, $r\mapsto u(r)$ defines a radially symmetric stationary solution of the parabolic system \cref{parabolic_system_higher_space_dimension} if and only if it satisfies, on $(0,+\infty)$, the (non-autonomous) differential system
\begin{equation}
\label{differential_system_radially_symmetric_stationary_order_2}
\ddot u (r) = -\frac{\dSpace-1}{r} \dot u(r) + \nabla V\bigl(u(r)\bigr)
\,,
\end{equation} 
where $\dot u$ and $\ddot u$ stand for the first and second derivatives of $r\mapsto u(r)$, together with the limit
\begin{equation}
\label{boundary_condition_at_r_equals_0_order_2}
\dot u(r) \to 0 \quad\text{as}\quad r\to0^+ 
\,.
\end{equation}
Observe that, in this case, $u(\cdot)$ is actually the restriction to $[0,+\infty)$ of an even function in $\ccc^3(\rr,\rr^\dState)$ which is a solution (on $\rr$) of the differential system \cref{differential_system_radially_symmetric_stationary_order_2} (the limit \cref{boundary_condition_at_r_equals_0_order_2} ensures that equality \cref{differential_system_radially_symmetric_stationary_order_2} still makes sense and holds at $r$ equals $0$). In other words, provided that condition \cref{boundary_condition_at_r_equals_0_order_2} holds, it is equivalent to assume that system \cref{differential_system_radially_symmetric_stationary_order_2} holds on $(0,+\infty)$ or on $[0,+\infty)$. By abuse of language, the terminology \emph{radially symmetric stationary solution of system \cref{parabolic_system_higher_space_dimension}} will refer, all along the paper, to functions $u:[0,+\infty)\to\rr^{\dState}$ satisfying these conditions \cref{differential_system_radially_symmetric_stationary_order_2,boundary_condition_at_r_equals_0_order_2} (even if, formally, it is rather the function $\rr^\dSpace\to\rr^\dState$, $x\mapsto u\bigl(\abs{x}\bigr)$ that fits with this terminology). 

Let us denote by $\Sigma_{\min}(V)$ the set of \emph{nondegenerate} (local or global) minimum points of $V$; with symbols, 
\[
\Sigma_{\min}(V) = \bigl\{u\in\rr^{\dState}: \nabla V(u) = 0 \text{ and } D^2V(u)>0\bigr\}
\,.
\]
Throughout all the paper, the words \emph{minimum point} will be used to denote a local \emph{or} global minimum point of a (potential) function. 
\begin{definition}
\label{def:stable_close_to_m_at_infinit}
A (global) solution $(0,+\infty)\to\rr^{\dState}$, $r\mapsto u(r)$, of the differential system \cref{differential_system_radially_symmetric_stationary_order_2} (in particular a radially symmetric stationary solution of system \cref{parabolic_system_higher_space_dimension}) is said to be \emph{stable at infinity} if $u(r)$ approaches a point of $\Sigma_{\min}(V)$ as $r$ goes to $+\infty$. If this point of $\Sigma_{\min}(V)$ is denoted by $u_\infty$, then the solution is said to be \emph{stable close to $u_\infty$ at infinity}. 
\end{definition}
\begin{notation}
For every $u_\infty$ in $\Sigma_{\min}(V)$, let $\ssss_{\VuInfty}$ denote the set of the radially symmetric stationary solutions of system \cref{parabolic_system_higher_space_dimension} that are stable close to $u_\infty$ at infinity. With symbols,
\[
\ssss_{\VuInfty} = \bigl\{
u:[0,+\infty)\to\rr^{\dState}:  u \text{ satisfies \cref{differential_system_radially_symmetric_stationary_order_2,boundary_condition_at_r_equals_0_order_2}} \text{ and $u(r)\xrightarrow[r\to+\infty]{}u_\infty$}
\bigr\}
\,.
\]
Let
\[
\ssss_{\VuInfty}^0 = \Bigl\{u(0): u\in\ssss_{\VuInfty} \Bigr\} 
\,,
\]
and let
\begin{equation}
\label{notation_sss_V}
\ssss_V = \bigsqcup_{u_\infty\in\Sigma_{\min}(V)} \ssss_{\VuInfty} 
\quad\text{and}\quad
\ssss_V^0 = \bigsqcup_{u_\infty\in\Sigma_{\min}(V)} \ssss_{\VuInfty}^0
\,. 
\end{equation}
\end{notation}
The following statement is an equivalent (simpler) formulation of conclusion \cref{item:cor_insight_discreteness} of \cref{cor:insight_main_result}.
\begin{corollary}
\label{cor:insight_main_result_bis}
For a generic potential $V$, the subset $\ssss_V^0$ of $\rr^{\dState}$ is discrete. 
\end{corollary}
\subsection{Differential systems governing radially symmetric stationary solutions}
\label{subsec:differential_systems_gov_rad_symm_stat_sol}
The second-order differential system \cref{differential_system_radially_symmetric_stationary_order_2} is equivalent to the (non-autonomous) $2\dState$-dimensional first order differential differential system
\begin{equation}
\label{differential_system_radially_symmetric_stationary_order_1}
\left\{
\begin{aligned}
\dot u &= v \\
\dot v &= -\frac{\dSpace-1}{r}v + \nabla V(u)
\,.
\end{aligned}
\right.
\end{equation}
Introducing the auxiliary variables $\tau$ and $c$ defined as
\begin{equation}
\label{tau_and_c}
\tau=\log(r)
\quad\text{and}\quad
c = \frac{1}{r}
\,,
\end{equation}
the previous $2\dState$-dimensional differential system \cref{differential_system_radially_symmetric_stationary_order_1} is equivalent to each of the following two $2\dState+1$-dimensional autonomous differential systems:
\begin{equation}
\label{differential_system_autonomous_time_tau}
\left\{
\begin{aligned}
u_\tau &= rv \\
v_\tau &= -(\dSpace-1)v + r \nabla V(u) \\
r_\tau &= r
\,,
\end{aligned}
\right.
\end{equation}
and
\begin{equation}
\label{differential_system_autonomous_time_r}
\left\{
\begin{aligned}
u_r &= v \\
v_r &= -(\dSpace-1)c v + \nabla V(u) \\
c_r &= -c^2 
\,.
\end{aligned}
\right.
\end{equation}
\begin{remark}
Integrating the third equations of systems \cref{differential_system_autonomous_time_tau,differential_system_autonomous_time_r} yields
\[
r = r_0e^{\tau-\tau_0}
\quad\text{and}\quad
\frac{1}{c}-\frac{1}{c_0} = r-r_0
\,,
\]
and the parameters $\tau_0$ and $c_0$ (which determine in each case the origin of ``time'') do not matter in principle, since those systems are autonomous. However, if the ``initial conditions'' $r_0$ and $c_0$ are positive (which is true for the solutions that describe radially symmetric stationary solutions of system \cref{parabolic_system_higher_space_dimension}), it is natural to choose, in each case, the origins of time according to equalities \cref{tau_and_c}, that is~: 
\[
\tau_0 = \ln(r_0)\quad\text{and}\quad c_0 = \frac{1}{r_0}
\,.
\]
\end{remark}
\begin{figure}[htbp]
\centering
\includegraphics[width=\textwidth]{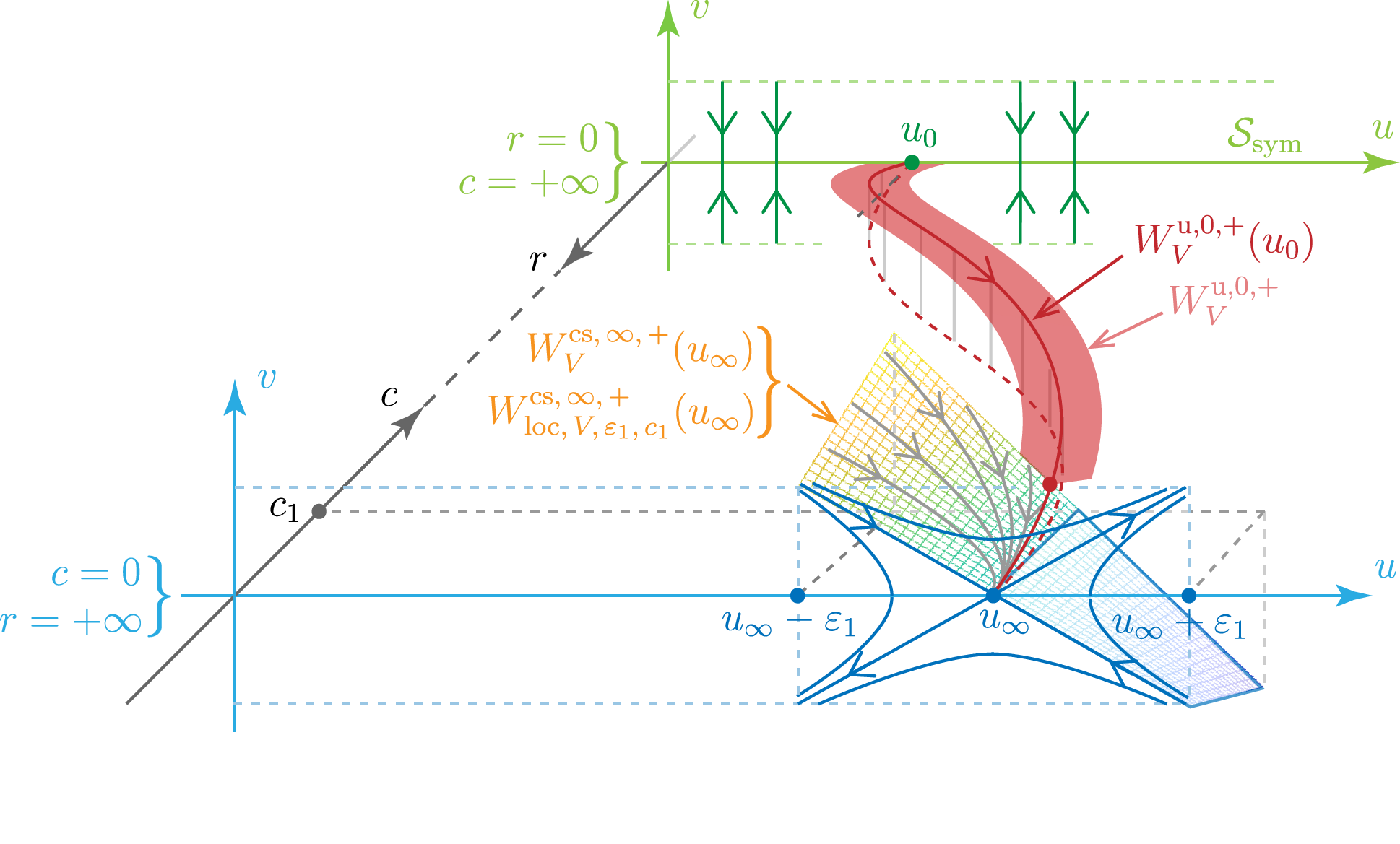}
\caption{Dynamics of the (equivalent) differential systems \cref{differential_system_autonomous_time_tau} (for $r$ nonnegative finite) and \cref{differential_system_autonomous_time_r} (for $c=1/r$ nonnegative finite) in $\rr^{\dState}\times\rr^{\dState}\times[0,+\infty]$ (this domain is three-dimensional if $\dState$ is equal to $1$, as on the figure). For the limit differential system \cref{differential_system_order_1_limit_r_equals_0} in the subspace $r=0$ (in green), the trajectories are vertical and the solutions converge towards the horizontal $u$-axis, defined as $\sssSym$ in \cref{notation_sssSym}, and which is the higher space dimensional analogue of the symmetry subspace for symmetric standing pulses in space dimension $1$. The point $u_\infty$ is a local minimum point of $V$, so that the point $(u_\infty,0_{\rr^{\dState}})$ is a hyperbolic equilibrium for the limit differential system \cref{differential_system_order_1_limit_r_goes_to_infty} in the subspace $c=0\iff r=+\infty$ (in blue). Systems \cref{differential_system_autonomous_time_tau,differential_system_autonomous_time_r} are autonomous, but the quantity $r$ (the quantity $c$) goes monotonously from $0$ to $+\infty$ (from $+\infty$ to $0$) for all the solutions in the subspace $r>0\iff c>0$, so that those solutions can be parametrized with $r$ (with $c$) as time. The unstable manifold $\WuOriginPlus_V(u_0)$ is one-dimensional and is a transverse intersection between the unstable set $\WuOriginPlus_V$ of the subspace $\{r=0,v=0_{\rr^{\dState}}\}$ and the centre stable manifold $\WcsInftyPlus_V(u_\infty)$ of the equilibrium $(u_\infty,0_{\rr^{\dState}},c=0)$. To prove the generic transversality of this intersection is the main goal of the paper. The dotted red curve is the projection onto the $(u,r)$-subspace of this intersection. The part of $\WcsInftyPlus_V(u_\infty)$ which is displayed on the figure can also be seen as the \emph{local} centre stable manifold $\WcslocInftyPlus{V}{\varepsilon_1}{c_1}(u_\infty)$ defined in \cref{def_WcslocPlus} (with $u_\infty$ equal to the point $u_{\infty,1}$ introduced there).}
\label{fig:local_centre_stable_man_at_infinity}
\end{figure}
\paragraph*{Properties close to origin.}
System \cref{differential_system_autonomous_time_tau} is relevant to provide an insight into the limit system \cref{differential_system_radially_symmetric_stationary_order_1} as $r$ goes to $0$. The subspace $\rr^{2\dState}\times\{0\}$ ($r$ equal to $0$) is invariant by the flow of this system, and the system reduces on this invariant subspace to
\begin{equation}
\label{differential_system_order_1_limit_r_equals_0}
\left\{
\begin{aligned}
u_\tau &= 0 \\
v_\tau &= -(\dSpace-1)v 
\,,
\end{aligned}
\right.
\end{equation}
see \cref{fig:local_centre_stable_man_at_infinity}. For every $u_0$ in $\rr^{\dState}$, the point $(u_0,0_{\rr^{\dState}},0)$ is an equilibrium of system \cref{differential_system_autonomous_time_tau}; let us denote by $\WuOrigin_V(u_0)$ the (one-dimensional) unstable manifold of this equilibrium, for this system, let 
\begin{equation}
\label{def_WuOriginPlus_V_of_u0}
\WuOriginPlus_V(u_0) = \WuOrigin_V(u_0)\cap\bigl(\rr^{2\dState}\times(0,+\infty)\bigr)
\,,
\end{equation}
and let
\[
\WuOriginPlus_V = \bigsqcup_{u_0\in\rr^{\dState}} \WuOriginPlus_V(u_0)
\,.
\]
The subspace 
\begin{equation}
\label{notation_sssSym}
\sssSym = \rr^{\dState}\times\{0_{\rr^{\dState}}\}\times\{0\}
\end{equation}
of $\rr^{2\dState+1}$ can be seen as the higher space dimension analogue of the symmetry (reversibility) subspace $\rr^{\dState}\times\{0_{\rr^{\dState}}\}$ of $\rr^{2\dState}$ (which is relevant for symmetric standing pulses in space dimension $1$, see \cite{JolyRisler_genericTransversalityTravStandFrontsPulses_2023} and \cref{subsec:similarites_differences_with_standing_pulses} below); the set $\WuOriginPlus_V$ can be seen as the unstable manifold of this subspace $\sssSym$. 
\paragraph*{Properties close to infinity.}
System \cref{differential_system_autonomous_time_r} is relevant to provide an insight into the limit system \cref{differential_system_radially_symmetric_stationary_order_1} as $r$ goes to $+\infty$. The subspace $\rr^{2\dState}\times\{0\}$ of $\rr^{2\dState+1}$ ($c$ equal to $0$, or in other words $r$ equal to $+\infty$) is invariant by the flow of this system, and the system reduces on this invariant subspace to
\begin{equation}
\label{differential_system_order_1_limit_r_goes_to_infty}
\left\{
\begin{aligned}
u_r &= v \\
v_r &= \nabla V(u) 
\,.
\end{aligned}
\right.
\end{equation}
For every $u_\infty$ in $\Sigma_{\min}(V)$, the point $(u_\infty,0_{\rr^{\dState}},0)$ is an equilibrium of system \cref{differential_system_autonomous_time_r}; let us consider its global centre-stable manifold in $\rr^{2\dState}\times(0,+\infty)$, defined as
\begin{equation}
\label{def_WcsInftyPlus_of_uInfty}
\begin{aligned}
\WcsInftyPlus_V(u_\infty) &= \Bigl\{
(u_0,v_0,c_0)\in\rr^{2\dState}\times(0,+\infty): \text{ the solution of system \cref{differential_system_autonomous_time_r}}\\
&\text{with initial condition $(u_0,v_0,c_0)$ at ``time'' $r_0=1/c_0$ is}\\
&\text{defined up to $+\infty$ and goes to $(u_\infty,0,0)$ as $r$ goes to $+\infty$}
\Bigr\}
\,.
\end{aligned}
\end{equation}
This set $\WcsInftyPlus_V(u_\infty)$ is a $\dState+1$-dimensional submanifold of $\rr^{2\dState}\times(0,+\infty)$ (see \cref{subsec:behaviour_solutions_stable_at_infinity}). 
\paragraph*{Radially symmetric stationary solutions.}
Let us consider the involution
\[
\iota:\rr^{2\dState}\times(0,+\infty)\to \rr^{2\dState}\times(0,+\infty)
\,,\quad
(u,v,r)\mapsto (u,v,1/r)
\,.
\]
The following lemma, proved in \cref{subsec:proof_of_lem_equivalence_rad_sym_stat_stable_at_infty_intersection_Wu_Ws}, formalizes the correspondence between the radially symmetric stationary solutions stable at infinity for system \cref{parabolic_system_higher_space_dimension} and the manifolds defined above. 
\begin{lemma}
\label{lem:equivalence_rad_sym_stat_stable_at_infty_intersection_Wu_Ws}
Let $u_\infty$ be a point of $\Sigma_{\min}(V)$. A (global) solution $[0,+\infty)\to\rr^{\dState}$, $r\mapsto u(r)$ of system \cref{differential_system_radially_symmetric_stationary_order_2} belongs to $\ssss_{\VuInfty}$ if and only if its trajectory (in $\rr^{2\dState}\times(0,+\infty)$)
\begin{equation}
\label{trajectory_time_tau}
\Bigl\{\bigl(u(r),\dot u(r),r\bigr):r\in(0,+\infty)\Bigr\}
\end{equation}
belongs to the intersection
\begin{equation}
\label{intersection_unstable_stable_manifolds}
\WuOriginPlus_V\cap \iota^{-1}\bigl(\WcsInftyPlus_V(u_\infty)\bigr)
\,.
\end{equation}
\end{lemma}
\subsection{Transversality of radially symmetric stationary solutions stable at infinity}
\begin{definition}
\label{def:transversality_rad_symm_solution_stable_at_infinity}
Let $u_\infty$ be a point of $\Sigma_{\min}(V)$. A radially symmetric stationary solution stable close to $u_\infty$ at infinity for system \cref{parabolic_system_higher_space_dimension} (in other words, a function $u$ of $\ssss_{\VuInfty}$) is said to be \emph{transverse} if the intersection \cref{intersection_unstable_stable_manifolds} is transverse, in $\rr^{2\dState}\times(0,+\infty)$, along the trajectory \cref{trajectory_time_tau}. 
\end{definition}  
\begin{remark}
The natural analogue of radially symmetric stationary solutions stable at infinity when space dimension $\dSpace$ is equal to $1$ are symmetric standing pulses stable at infinity (see \mainPaperDefSymmetricStandingPulse{} of \cite{JolyRisler_genericTransversalityTravStandFrontsPulses_2023}), and the natural analogue for such pulses of \cref{def:transversality_rad_symm_solution_stable_at_infinity} above is their \emph{elementarity}, not their transversality (see \mainPaperDefTransversePulse{} and \mainPaperDefElementarySymmetricStandingPulse{} of \cite{JolyRisler_genericTransversalityTravStandFrontsPulses_2023}). However, the transversality of a symmetric standing pulse (when the space dimension $\dSpace$ equals $1$) makes little sense in higher space dimension, because of the singularity at $r$ equals $0$ for the differential systems \cref{differential_system_radially_symmetric_stationary_order_2,differential_system_radially_symmetric_stationary_order_1}, or because of the related fact that the subspace $\{r=0\}$ is invariant for the differential system \cref{differential_system_autonomous_time_tau}. For that reason, the adjective \emph{transverse} (not \emph{elementary}) is chosen to qualify the property considered in \cref{def:transversality_rad_symm_solution_stable_at_infinity} above.
\end{remark}
\subsection{The space of potentials}
\label{subsec:space_of_potentials}
For the remaining of the paper, let us take and fix an integer $k$ not smaller than $1$. Let us consider the space $\CkbFull$ of functions $\rr^d\to\rr$ of class $\ccc^{k+1}$ which are bounded, as well as their derivatives of order not larger than $k+1$, equipped with the norm
\[
\norm{W}_{\CkbShort} = \max_{\alpha\text{ multi-index, }|\alpha|\leq k+1} \|\partial^{|\alpha|}_{u_\alpha} W\|_{L^\infty(\rr^d,\rr)}
\,, 
\]
and let us embed the larger space $\CkFull$ with the following topology: for $V$ in this space, a basis of neighbourhoods of $V$ is given by the sets $V+\ooo$, where $\ooo$ is an open subset of $\CkbFull$ embedded with the topology defined by $\norm{\cdot}_{\CkbShort}$ (which can be viewed as an extended metric). For comments concerning the choice of this topology, see \mainPaperSubsectionSpaceOfPotentials{} of \cite{JolyRisler_genericTransversalityTravStandFrontsPulses_2023}. 
\subsection{Main result}
The following generic transversality statement is the main result of this paper. 
\begin{theorem}[generic transversality of radially symmetric stationary solutions stable at infinity]
\label{thm:main}
There exists a generic subset $\gggg$ of $\left(\CkFull,\normCkb\right)$ such that, for every potential function $V$ in $\gggg$, every radially symmetric stationary solution stable at infinity of the parabolic system \cref{parabolic_system_higher_space_dimension} is transverse. 
\end{theorem}
\Cref{thm:main} can be viewed as the extension to higher space dimensions (for radially symmetric solutions) of conclusion 2 of \mainPaperMainTheorem{} of \cite{JolyRisler_genericTransversalityTravStandFrontsPulses_2023} (which is concerned with elementary standing pulses stable at infinity in space dimension $1$). A short comparison between these two results and their proofs is provided in the next \namecref{subsec:similarites_differences_with_standing_pulses}. For more comments and a short historical review on transversality results in similar contexts, see \mainPaperSubsectionHistoricalReview{} of the same reference.

The core of the paper (\cref{sec:generic_transversality_potentials_quadratic_past_given_radius}) is devoted to the proof of the conclusions of \cref{thm:main} among potentials which are quadratic past a certain radius (defined in \cref{notation_vvvQuad_of_R}), as stated in \cref{prop:generic_transversality_potentials_quadratic_past_given_radius}. The extension to general potentials of $\CkbFull$ is carried out in \cref{sec:proof_main_results}.
\begin{remark}
As in \cite{JolyRisler_genericTransversalityTravStandFrontsPulses_2023} (see \mainPaperTheoremExtensionEigenspaceSmallestEigenvalueAtInfty{} of that reference), the same arguments could be called upon to prove that the following additional conclusions hold, generically with respect to the potential $V$: 
\begin{enumerate}
\item for every minimum point of $V$, the smallest eigenvalue of $D^2V$ at this minimum point is simple;
\item every radially symmetric stationary solution stable at infinity of the parabolic system \cref{parabolic_system_higher_space_dimension} approaches its limit at infinity tangentially to the eigenspace corresponding to the smallest eigenvalue of $D^2V$ at this point. 
\end{enumerate}
\end{remark}
\subsection{Key differences with the generic transversality of standing pulses in space dimension one}
\label{subsec:similarites_differences_with_standing_pulses}
\Cref{table:comparison} lists the key differences between the proof of the generic elementarity of symmetric standing pulses carried out in \cite{JolyRisler_genericTransversalityTravStandFrontsPulses_2023}, and the proof of the generic transversality of radially symmetric stationary solutions carried out in the present paper (implicitly, the other steps/features of the proofs are similar or identical). The state dimension, which is simply denoted by $d$ in \cite{JolyRisler_genericTransversalityTravStandFrontsPulses_2023}, is here denoted by $\dState$ in both cases. Some of the notation/rigour is lightened. 
\begin{table}[htbp]
\centering
{\footnotesize
\begin{tabular}{|p{.25\textwidth}|p{.31\textwidth}|p{.35\textwidth}|}
\hline
& \makecell{\textbf{Symmetric standing pulse}} & \makecell{\textbf{Radially symmetric} \\ \textbf{stationary solution}} \\ \hline
\makecell[l]{Critical point at infinity} & \makecell[l]{critical point $e$, $E=(e,0_{\rr^{\dState}})$} & \makecell[l]{minimum point $u_\infty$} \\ \hline
\makecell[l]{Symmetry subspace $\sssSym$} & \makecell[l]{$\{(u,v)\in\rr^{2\dState}:v=0\}$, \\ dimension $\dState$} & \makecell[l]{$\{(u,v,r)\in\rr^{2\dState+1}:(v,r)=(0,0)\}$, \\ dimension $\dState$} \\ \hline
\makecell[l]{Differential system \\ governing the profiles} & \makecell[l]{autonomous, conservative, \\ regular at $\sssSym$} & \makecell[l]{non-autonomous, dissipative, \\singular at reversibility subspace} \\ \hline
\makecell[l]{Direction of the flow} & \makecell[l]{$E\to\sssSym$} & \makecell[l]{$\sssSym\to u_\infty$} \\ \hline
\makecell[l]{Invariant manifold at \\ infinity} & \makecell[l]{$\Wu(E)$, dimension $\dState-m(e)$} & \makecell[l]{$\WcsInftyPlus(u_\infty)$, dimension $\dState+1$} \\ \hline
\makecell[l]{Invariant manifold at \\symmetry subspace} & \makecell[l]{none} & \makecell[l]{$\WuOriginPlus$, dimension $\dState+1$} \\ \hline
\makecell[l]{Transversality} & \makecell[l]{$\Wu(E)\pitchfork\sssSym$} & \makecell[l]{$\WcsInftyPlus(u_\infty)\pitchfork\WuOriginPlus$} \\ \hline
\makecell[l]{Transversality of spatially \\ homogeneous solutions} & \makecell[l]{irrelevant} & \makecell[l]{\cref{prop:transversality_of_homogeneous_solutions}} \\ \hline
\makecell[l]{Interval $\IOnce$ (values \\ reached only once)} & \makecell[l]{``anywhere''} & \makecell[l]{close to $\sssSym$} \\ \hline
\makecell[l]{$\mmm$ (departure set of $\Phi$)} & \makecell[l]{parametrization of $\partial\WulocZero{V}(E)$ \\ and time, dimension $\dState-m(e)$} & \makecell[l]{$\sssSym$ and $\WcslocInftyPlusMinimized(u_\infty)$ at $r=N$, \\ dimension $2\dState$} \\ \hline
\makecell[l]{$\nnn$ (arrival set of $\Phi$)} & \makecell[l]{$\rr^{2\dState}$} & \makecell[l]{$\rr^{2\dState}\times\rr^{2\dState}$} \\ \hline
\makecell[l]{$\www$ (target manifold)} & \makecell[l]{$\sssSym$} & \makecell[l]{diagonal of $\nnn$} \\ \hline
\makecell[l]{$\dim(\mmm)-\codim(\www)$} & \makecell[l]{$-m(e)$} & \makecell[l]{$0$} \\ \hline
\makecell[l]{Condition to be fulfilled \\ by perturbation $W$} & \makecell[l]{$\left\langle D\Phi(W)\bigm|(0,\psi)\right\rangle\not=0$} & \makecell[l]{$\left\langle D\Phiu(W)\bigm|(\phi,\psi)\right\rangle\not=0$} \\ \hline
\makecell[l]{Perturbation $W$, case 3} & \makecell[l]{precluded} & \makecell[l]{$W(u_0)\not=0$} \\ \hline
\end{tabular}
}
\caption{Formal comparison between the generic elementarity of symmetric standing pulses (space dimension $1$) proved in \cite{JolyRisler_genericTransversalityTravStandFrontsPulses_2023}, and the generic transversality of radially symmetric stationary solutions (higher space dimension $\dSpace$) proved in the present paper.}
\label{table:comparison}
\end{table}

Here are a few additional comments about these differences. 

Concerning the critical point at infinity, $u_\infty$ is assumed (here) to be a minimum point, whereas (in \cite{JolyRisler_genericTransversalityTravStandFrontsPulses_2023}) the Morse index of $e$ is any. Indeed, if the Morse index $m(u_\infty)$ of $u_\infty$ was positive, then the dimension of the centre-stable manifold $\WcsInftyPlus_V(u_\infty)$ would be equal to $\dState+m(u_\infty)+1$; as a consequence, proving the transversality of the intersection \cref{intersection_unstable_stable_manifolds} in that case would require more stringent regularity assumptions on $V$ (see hypothesis \mainPaperItemDimensionsSardSmaleTheorem{} of \mainPaperSardSmaleTheorem{} of \cite{JolyRisler_genericTransversalityTravStandFrontsPulses_2023}) while nothing particularly useful could be derived from this transversality. On the other hand, assuming that $u_\infty$ is a minimum point allows to view its local centre-stable manifold as a graph $(u,c)\mapsto v$ (see \cref{prop:loc_centre_stab_manifold_at_infinity}), which is slightly simpler. 

Concerning the interval $\IOnce$ providing values $u$ reached ``only once'' by the profile (\cref{lem:dot_u_does_not_vanish_and_values_reached_once}), the proof of the present paper takes advantage of the dissipation to find a convenient interval close to the ``departure point'' $u_0$, as was done in \cite{JolyRisler_genericTransversalityTravStandFrontsPulses_2023} for travelling fronts (whereas, for standing pulse, the interval is to be found ``anywhere'', thanks to the conservative nature of the differential system governing the profiles, see conclusion \mainPaperItemTFPropValuesReachedOnlyOnce{} of \mainPaperPropValuesReachedOnlyOnce{} of \cite{JolyRisler_genericTransversalityTravStandFrontsPulses_2023}).

Concerning the function $\Phi$ to which Sard--Smale theorem is applied in the present paper, both manifolds $\WuOriginPlus$ and $\WcsInftyPlus(u_\infty)$ depend on the potential $V$. However, the transversality of an intersection between these two manifolds can be seen as the transversality of the image of $\Phi$ with the (fixed) diagonal of $\rr^{2\dState}\times\rr^{2\dState}$, for a function $\Phi$ combining the parametrization of these two manifolds. This trick, which is the same as in \cite{JolyRisler_genericTransversalityTravStandFrontsPulses_2023} for travelling fronts, allows to apply Sard--Smale theorem to a function $\Phi$ with a fixed arrival space $\nnn$ containing a fixed target manifold $\www$ (in this case the diagonal of $\nnn$). By contrast, for symmetric standing pulses in \cite{JolyRisler_genericTransversalityTravStandFrontsPulses_2023}, since the subspace $\sssSym$ involved in the transverse intersection is fixed, the previous trick is unnecessary and the setting is simpler. 

Finally, a technical difference occurs in ``case 3'' of the proof that the degrees of freedom provided by perturbing the potential allow to reach enough directions in the arrival state of $\Phi$ (\cref{lem:existence_suitable_perturbation_direction_W}, which is the core of the proof). In \cite{JolyRisler_genericTransversalityTravStandFrontsPulses_2023}, case 3 is shown to lead to a contradiction, not only for symmetric standing pulses, but also for asymmetric ones and for travelling fronts. Here, such a contradiction does not seem to occur (or at least is more difficult to prove), but this has no harmful consequence: a suitable perturbation of the potential can still be found in this case. 
\section{Preliminary properties}
\subsection{Proof of \texorpdfstring{\cref{lem:equivalence_rad_sym_stat_stable_at_infty_intersection_Wu_Ws}}{Lemma \ref{lem:equivalence_rad_sym_stat_stable_at_infty_intersection_Wu_Ws}}}
\label{subsec:proof_of_lem_equivalence_rad_sym_stat_stable_at_infty_intersection_Wu_Ws}
Let $V$ denote a potential function in $\CkFull$. Let $(0,+\infty)\to\rr^{\dState}$, $r\mapsto u(r)$ denote a (global) solution of system \cref{differential_system_radially_symmetric_stationary_order_2}, assumed to be stable close to some point $u_\infty$ of $\Sigma_{\min}(V)$ at infinity (\cref{def:stable_close_to_m_at_infinit}). \Cref{lem:equivalence_rad_sym_stat_stable_at_infty_intersection_Wu_Ws} follows from the next lemma. 
\begin{lemma}
\label{lem:derivative_goes_to_0_when_r_goes_to_infty}
The derivative $\dot u(r)$ goes to $0$ as $r$ goes to $+\infty$. 
\end{lemma}
\begin{proof}
Let us consider the Hamiltonian function
\begin{equation}
\label{Hamiltonian}
H_V : \rr^{2\dState} \to \rr
\,,\quad
(u,v) \mapsto \frac{v^2}{2} - V(u)
\,,
\end{equation}
and, for every $r$ in $(0,+\infty)$, let
\[
h(r) = H_V\bigl(u(r),\dot u(r)\bigr)
\,.
\]
It follows from system \cref{differential_system_radially_symmetric_stationary_order_2} that, for every $r$ in $(0,+\infty)$, 
\begin{equation}
\label{Ham_decrase}
\dot h(r) = - \frac{\dSpace-1}{r}\dot u(r)^2
\,,
\end{equation}
thus the function $h(\cdot)$ decreases, and it follows from the expression \cref{Hamiltonian} of the Hamiltonian that this function converges, as $r$ goes to $+\infty$, towards a finite limit $h_\infty$ which is not smaller than $-V(u_\infty)$. 

Let us proceed by contradiction and assume that $h_\infty$ is larger than $-V(u_\infty)$. Then, it follows again from the expression \cref{Hamiltonian} of the Hamiltonian that the quantity $\dot u(r)^2$ converges towards the positive quantity $2\bigl(h_\infty+V(u_\infty)\bigr)$ as $r$ goes to $+\infty$. As a consequence, it follows from equality \cref{Ham_decrase} that $h(r)$ goes to $-\infty$ as $r$ goes to $+\infty$, a contradiction. \Cref{lem:derivative_goes_to_0_when_r_goes_to_infty} is proved. 
\end{proof}
\subsection{Transversality of homogeneous radially symmetric stationary solutions stable at infinity}
\begin{proposition}
\label{prop:transversality_of_homogeneous_solutions}
For every potential function $V$ in $\CkFull$ and for every nondegenerate minimum point $u_\infty$ of $V$, the constant function
\[
[0,+\infty)\to\rr^{\dState}\,,\quad r\mapsto u_\infty
\,,
\]
which defines an (homogeneous) radially symmetric stationary solution stable at infinity for system \cref{parabolic_system_higher_space_dimension}\,, is transverse (in the sense of \cref{def:transversality_rad_symm_solution_stable_at_infinity}). 
\end{proposition}
\begin{proof}
Let $V$ denote a function in $\CkFull$ and $u_\infty$ denote a nondegenerate minimum point of $V$. The function $[0,+\infty)\to\rr^{\dState}$, $r\mapsto u_\infty$ is a (constant) solution of the differential system \cref{differential_system_radially_symmetric_stationary_order_1}, and the linearization of this differential system around this solution reads
\begin{equation}
\label{linearized_system_along_u_infty}
\ddot u = -\frac{\dSpace-1}{r} \dot u + D^2V(u_\infty)\cdot u
\,.
\end{equation}
Let $(0,+\infty)\to\rr^{\dState}$, $r\mapsto u(r)$ denote a nonzero solution of this differential system, and, for every $r$ in $(0,+\infty)$, let 
\[
v(r) = \dot u(r)
\quad\text{and}\quad
U(r) = \bigl(u(r),v(r)\bigr)
\quad\text{and}\quad
q(r) = \frac{u(r)^2}{2}
\,.
\]
Then (omitting the dependency on $r$), 
\[
\dot q = u\cdot\dot u 
\quad\text{and}\quad
\ddot q = \dot u^2 + u\cdot\ddot u = \dot u^2 - \frac{\dSpace-1}{r} \dot q + D^2V(u_\infty)\cdot(u,u)
\,,
\]
so that
\[
\frac{d}{dr}\bigl(r^{\dSpace-1}\dot q(r)\bigr) = r^{\dSpace-1}\left(\ddot q + \frac{\dSpace-1}{r} \dot q \right) = r^{\dSpace-1}\bigl(\dot u^2+ D^2V(u_\infty)\cdot(u,u)\bigr)
\,.
\]
Since $r\mapsto u(r)$ was assumed to be nonzero, it follows that the quantity $r^{\dSpace-1}\dot q(r)$ is strictly increasing on $(0,+\infty)$. To prove the intended conclusion, let us proceed by contradiction and assume that, for every $r$ in $(0,+\infty)$, $\bigl(u(r),v(r),r\bigr)$ belongs: 
\begin{enumerate}
\item to the tangent space $T_{(u_\infty,0_{\rr^{\dState}},r)} \WuOriginPlus_V(u_\infty)$, 
\label{item:u_dot_u_in_tangent_space_to_unstable_manifold}
\item \emph{and} to the tangent space $T_{(u_\infty,0_{\rr^{\dState}},r)} \Bigl(\iota^{-1}\bigl(\WcsInftyPlus_V(u_\infty)\bigr)\Bigr)$. 
\label{item:u_dot_u_in_tangent_space_to_centre_stable_manifold}
\end{enumerate}
As in \cref{tau_and_c}, let us introduce the auxiliary variables $\tau$ (equal to $\log(r)$) and $c$ (equal to $1/r$). With this notation, system \cref{linearized_system_along_u_infty} is equivalent to 
\begin{equation}
\label{differential_system_autonomous_time_tau_linearized_u_infty}
\left\{
\begin{aligned}
u_\tau &= rv \\
v_\tau &= -(\dSpace-1)v + r D^2V(u_\infty)\cdot u \\
r_\tau &= r
\,,
\end{aligned}
\right.
\end{equation}
and to
\begin{equation}
\label{differential_system_autonomous_time_r_linearized_u_infty}
\left\{
\begin{aligned}
u_r &= v \\
v_r &= -(\dSpace-1)c v + D^2V(u_\infty)\cdot u \\
c_r &= -c^2
\,.
\end{aligned}
\right.
\end{equation}
Assumptions \cref{item:u_dot_u_in_tangent_space_to_unstable_manifold,item:u_dot_u_in_tangent_space_to_centre_stable_manifold} above yield the following conclusions.
\begin{enumerate}
\item In view of the limit of system \cref{differential_system_autonomous_time_tau_linearized_u_infty} as $r$ goes to $0^+$, it follows from assumption \cref{item:u_dot_u_in_tangent_space_to_unstable_manifold} that there exists $\delta u_0$ in $\rr^{\dState}$ such that $\bigl(u(r),v(r)\bigr)$ goes to $(\delta u_0,0_{\rr^{\dState}})$ as $r$ goes to $0^+$;
\item and in view of the limit of system \cref{differential_system_autonomous_time_r_linearized_u_infty} as $c$ goes to $0^+$, it follows from assumption \cref{item:u_dot_u_in_tangent_space_to_centre_stable_manifold} that $\bigl(u(r),v(r)\bigr)$ goes to $(0_{\rr^{\dState}},0_{\rr^{\dState}})$, at an exponential rate, as $r$ goes to $+\infty$. 
\end{enumerate}
It follows from these two conclusions that the quantity $r^{\dSpace-1}\dot q(r)$ goes to $0$ as $r$ goes to $0^+$ \emph{and} as $r$ goes to $+\infty$, a contradiction with the fact (observed above) that this quantity is strictly increasing with $r$. \Cref{prop:transversality_of_homogeneous_solutions} is proved. 
\end{proof}
\subsection{Additional properties close to the origin}
\label{subsec:behaviour_close_to_origin}
Let $V$ denote a potential function in $\CkFull$ and let $u_0$ be a point in $\rr^{\dState}$. Let us recall (see \cref{subsec:differential_systems_gov_rad_symm_stat_sol}) that the unstable manifold $\WuOrigin_V(u_0)$ of the equilibrium $(u_0,0_{\rr^{\dState}},0)$ for the autonomous differential system \cref{differential_system_autonomous_time_tau}) is one-dimensional. As a consequence there exists a unique solution $r\mapsto u(r)$ of the differential system \cref{differential_system_radially_symmetric_stationary_order_2} such that the image of the map $r\mapsto\bigl(u(r),\dot u(r),r)$ lies in the intersection $\WuOriginPlus_V(u_0)$ of this unstable manifold with the half-space where $r$ is positive (this intersection was defined in \cref{def_WuOriginPlus_V_of_u0}); or, in other words, such that $\bigl(u(r),\dot u(r)\bigr)$ goes to $(u_0,0)$ as $r$ goes to $0^+$. This solution is defined on some (maximal) interval $(0,r_{\max})$, where $r_{\max}$ is either a finite quantity or $+\infty$. The following lemma provides properties of this solution that will be used in the sequel. To ease its statement, let us assume that $r_{\max}$ is equal to $+\infty$ (only this case will turn out to be relevant), and let us consider the continuous extension of $u(\cdot)$ to the interval $[0,+\infty)$ (and let us still denote by $u(\cdot)$ this continuous extension).
\begin{lemma}
\label{lem:dot_u_does_not_vanish_and_values_reached_once}
If $u(\cdot)$ is not identically equal to $u_0$ (in other words, if $u_0$ is not a critical point of $V$), then there exists a positive quantity $\rOnce$ such that, denoting by $\IOnce$ the interval $[0,\rOnce)$, the following conclusions hold:
\begin{enumerate}
\item the function $\dot u(\cdot)$ does not vanish on $\IOnce$,
\label{item:dot_u_does_not_vanish}
\item and, for every $r^*$ in $\IOnce$ and $r$ in $[0,+\infty)$, 
\[
u(r) = u(r^*)\implies r = r^*.
\]
\label{item:values_reached_once}
\end{enumerate}
\end{lemma}
\begin{proof}
The linearized system \cref{differential_system_autonomous_time_tau} at the equilibrium $(u_0,0_{\rr^{\dState}},0)$ reads:
\[
\frac{d}{d\tau}\begin{pmatrix}\delta u\\ \delta v\\ \delta r \end{pmatrix} = \begin{pmatrix} 0 & 0 & 0 \\ 0 & -(\dSpace-1) & \nabla V(u_0) \\ 0 & 0 & 1 \end{pmatrix} \begin{pmatrix}\delta u\\ \delta v\\ \delta r \end{pmatrix}
\,,
\]
thus the tangent space at $(u_0,0_{\rr^{\dState}},0)$ to $\WuOrigin_V(u_0)$ (the unstable eigenspace of the matrix of this system) is spanned by the vector $\bigl(0,\nabla V(u_0)/\dSpace,1\bigr)$; it follows that
\begin{equation}
\label{estimate_dot_u_close_to_origin}
\dot u(r) = \frac{r}{\dSpace}\nabla V(u_0)\bigl(1+o_{r\to0^+}(r)\bigr)
\,.
\end{equation}
Thus, if $r_0$ is a sufficiently small positive quantity, then $\dot u(\cdot)$ does not vanish on $(0,r_0]$ (so that conclusion \cref{item:dot_u_does_not_vanish} of \cref{lem:dot_u_does_not_vanish_and_values_reached_once} holds provided that $\rOnce$ is not larger than $r_0$), and the map
\begin{equation}
\label{u_of_r_between_0_and_r0}
[0,r_0]\to\rr^{\dState}\,,\quad r\mapsto u(r)
\end{equation}
is a $\ccc^1$-diffeomorphism onto its image. For $r$ in $[0,+\infty)$, let us denote $\bigl(u(r),\dot u(r)\bigr)$ by $U(r)$. According to the decrease \cref{Ham_decrase} of the Hamiltonian, there exists a quantity $\rOnce$ in $(0,r_0)$ such that, for every $r^*$ in $[0,\rOnce)$, 
\begin{equation}
\label{minus_V_at_rStar_larger_than_H_at_r0}
H_V\bigl(U(r_0)\bigr) < - V\bigl(u(r^*)\bigr)
\,.
\end{equation}
Take $r^*$ in $[0,\rOnce]$ and $r$ in $[0,+\infty)$, and let us assume that $u(r)$ equals $u(r^*)$. If $r$ was larger than $r_0$ then it would follow from the expression \cref{Hamiltonian} of the Hamiltonian, its decrease \cref{Ham_decrase}, and inequality \cref{minus_V_at_rStar_larger_than_H_at_r0} that
\[
-V\bigl(u(r)\bigr) \le H_V\bigl(U(r)\bigr) \le H_V\bigl(U(r_0)\bigr) < -V\bigl(u(r^*)\bigr)
\,,
\]
a contradiction with the equality of $u(r)$ and $u(r^*)$. Thus $r$ is not larger than $r_0$, and it follows from the one-to-one property of the function \cref{u_of_r_between_0_and_r0} that $r$ must be equal to $r^*$; conclusion \cref{item:values_reached_once} of \cref{lem:dot_u_does_not_vanish_and_values_reached_once} thus holds, and \cref{lem:dot_u_does_not_vanish_and_values_reached_once} is proved. 
\end{proof}
\subsection{Additional properties close to infinity}
\label{subsec:behaviour_solutions_stable_at_infinity}
Let $V_1$ denote a potential function in $\CkFull$ and $u_{1,\infty}$ denote a nondegenerate minimum point of $V_1$. According to the implicit function theorem, there exists a (small) neighbourhood $\nuRobust(V_1,u_{1,\infty})$ of $\vvvQuad{R}$ and a $\ccc^k$-function $V\mapsto u_\infty(V)$ defined on $\nuRobust(V_1,u_{1,\infty})$ and with values in $\rr^{\dState}$ such that $u_\infty(V_1)$ equals $u_{1,\infty}$ and, for every $V$ in $\nuRobust(V_1,u_{1,\infty})$, $u_\infty(V)$ is a local minimum point of $V$. The following proposition is nothing but the local centre-stable manifold theorem applied to the equilibrium $\bigl(u_\infty(V),0_{\rr^{\dState}},0\bigr)$ of the (autonomous) differential system \cref{differential_system_autonomous_time_r}, for $V$ close to $V_1$. Additional comments and references concerning local stable/centre/unstable manifolds are provided in \mainPaperSubsectionLocStabUnstMancPos{} of \cite{JolyRisler_genericTransversalityTravStandFrontsPulses_2023}. 
\begin{proposition}[local centre-stable manifold at infinity]
\label{prop:loc_centre_stab_manifold_at_infinity}
There exist a neighbourhood $\nu$ of $V_1$ in $\CkFull$, included in $\nuRobust(V_1,u_{1,\infty})$, such that, if $\varepsilon_1$ and $c_1$ denote sufficiently small positive quantities, then, for every $V$ in $\nu$, there exists a $\ccc^k$-map
\begin{equation}
\label{notation_wcsloc_of_V}
\wcslocInftyPlus{V}:\widebar{B}_{\rr^{\dState}}(u_{1,\infty},\varepsilon_1)\times[0,c_1]\to\rr^{\dState}
\,,\quad
(u,c)\mapsto \wcslocInftyPlus{V}(u,c)
\,,
\end{equation}
such that, for every $(u_0,v_0,c_0)$ in $\widebar{B}_{\rr^{\dState}}(u_{1,\infty},\varepsilon_1)\times\rr^{\dState}\times[0,c_1]$, the following two statements are equivalent:
\begin{enumerate}
\item $v = \wcslocInftyPlus{V}(u,c)$;
\item the solution $r\mapsto\bigl(u(r),v(r),c(r)\bigr)$ of the differential system \cref{differential_system_autonomous_time_r} with initial condition $(u_0,v_0,c_0)$ at time $r_0=1/c_0$ is defined up to $+\infty$, remains in $\widebar{B}_{\rr^{\dState}}(u_{1,\infty},\varepsilon_1)\times\rr^{\dState}\times[0,c_1]$ of all $r$ larger than $r_0$, and goes to $\bigl(u_\infty(V),0_{\rr^{\dState}},0\bigr)$ as $r$ goes to $+\infty$.
\end{enumerate}
In particular, $\wcslocInftyPlus{V}\bigl(u_\infty(V),0\bigr)$ is equal to $0_{\rr^{\dState}}$. In addition, the map
\[
\widebar{B}_{\rr^{\dState}}(u_{1,\infty},\varepsilon_1)\times[0,c_1]\times\nu\to\rr^{\dState}\,,\quad
(u,c,V)\mapsto \wcslocInftyPlus{V}(u,c)
\]
is of class $\ccc^k$ (with respect to $u$ and $c$ \emph{and} $V$), and, for every $V$ in $\nu$, the graph of the differential at $\bigl(u_\infty(V),0)$ of the map $(u,c)\mapsto \wcslocInftyPlus{V}(u,c)$ is equal to the centre-stable subspace of the linearization at $\bigl(u_\infty(V),0_{\rr^{\dState}},0\bigr)$ of the differential system \cref{differential_system_autonomous_time_r}. 
\end{proposition}
Let us denote by $\WcslocInftyPlus{V}{\varepsilon_1}{c_1}\bigl(u_\infty(V)\bigr)$ the graph of the map \cref{notation_wcsloc_of_V} (restricted to positive values of $c$), see \cref{fig:local_centre_stable_man_at_infinity}; with symbols, 
\begin{equation}
\label{def_WcslocPlus}
\WcslocInftyPlus{V}{\varepsilon_1}{c_1}\bigl(u_\infty(V)\bigr) = \Bigl\{\bigl(u,\wcslocInftyPlus{V}(u,c),c\bigr):(u,c)\in\widebar{B}_{\rr^{\dState}}(u_{1,\infty},\varepsilon_1)\times(0,c_1]\Bigr\}
\,.
\end{equation}
This set defines a local centre-manifold (restricted to positive values of $c$) for the equilibrium $\bigl(u_\infty(V),0_{\rr^{\dState}},0\bigr)$ of the differential system \cref{differential_system_autonomous_time_r}. Its uniqueness (for positive values of $c$) is ensured by the dynamics of the centre component $c$, which, according to the third equation of system \cref{differential_system_autonomous_time_r}, decreases to $0$ (see \cref{fig:local_centre_stable_man_at_infinity}).
The global centre-stable manifold $\WcsInftyPlus_V\bigl(u_\infty(V)\bigr)$ already defined in \cref{def_WcsInftyPlus_of_uInfty} can be redefined as the points of $\rr^{2\dState}\times(0,+\infty)$ that eventually reach the local centre manifold $\WcslocInftyPlus{V}{\varepsilon_1}{c_1}\bigl(u_\infty(V)\bigr)$ when they are transported by the flow of the differential system \cref{differential_system_autonomous_time_r}. 
\begin{remark}
If the state dimension $\dState$ is equal to $1$, then a calculation shows that
\[
\wcslocInftyPlus{V}(u,c) = -\bigl(u-u_\infty(V)\bigr)\left(\sqrt{V''\bigl(u_\infty(V)\bigr)} + \frac{\dSpace-1}{2}c+\dots\right)
\,,
\]
where ``$\dots$'' stands for higher order terms in $u-u_\infty(V)$ and $c$. In particular the quantity $\partial_c\partial_u\wcslocInftyPlus{V}\bigl(u_\infty(V),0\bigr)$ is equal to the (negative) quantity $-(\dSpace-1)/2$. The display of the local centre-stable manifold at infinity on \cref{fig:local_centre_stable_man_at_infinity} fits with the sign of this quantity. 
\end{remark}
\section{Tools for genericity}
Let 
\begin{equation}
\label{notation_vvvFull}
\vvvFull = \CkFull 
\,,
\end{equation}
and, for a positive quantity $R$, let
\begin{equation}
\label{notation_vvvQuad_of_R}
\vvvQuad{R} = \left\{ V\in\vvvFull : \text{for all $u$ in $\rr^d$}, \ \abs{u}\ge R \implies V(u) = \frac{u^2}{2} \right\} 
\,.
\end{equation}
Let us recall the notation $\ssss_V$ introduced in \cref{notation_sss_V}.
\begin{lemma}
\label{lem:global_flow_and_a_priori_bound_on_profiles}
For every positive quantity $R$ and for every potential $V$ in $\vvvQuad{R}$, the following conclusions hold. 
\begin{enumerate}
\item The flow defined by the differential system \cref{differential_system_radially_symmetric_stationary_order_2} (governing radially symmetric stationary solutions of the parabolic system \cref{parabolic_system_higher_space_dimension}) is global (that is, every solution is defined on $(0,+\infty)$).
\label{item:lem_global_flow}
\item For every $u$ in $\ssss_V$, the following bound holds:
\begin{equation}
\label{a_priori_bound_on_profiles}
\sup_{r\in(0,+\infty)} \abs{u(r)} < R
\,.
\end{equation}
\label{item:lem_a_priori_bound_on_profiles}
\end{enumerate}
\end{lemma}
\begin{proof}
Let $V$ be in $\vvvQuad{R}$. According to the definition \cref{notation_vvvQuad_of_R} of $\vvvQuad{R}$, there exists a positive quantity $K$ such that, for every $u$ in $\rr^{\dState}$, 
\[
\abs{\nabla V(u)}\le K + \abs{u}
\,.
\]
As a consequence, the following inequalities hold for the right-hand side of the first order differential system \cref{differential_system_radially_symmetric_stationary_order_1}:
\[
\abs{\left(v,-\frac{\dSpace-1}{r}v+\nabla V(u)\right)} 
\le \abs{v} + \frac{\dSpace-1}{r} \abs{v} + K + \abs{u}
\le K + \left(2+\frac{\dSpace-1}{r}\right)\abs{(u,v)}
\,,
\]
and this bound prevents the solution from blowing up in finite time, which proves conclusion \cref{item:lem_global_flow}.

Now, take a function $u$ in $\ssss_V$. Let us still denote by $u(\cdot)$ the continuous extension of this solution to $[0,+\infty)$. For every $r$ in $[0,+\infty)$, let 
\[
q(r) = \frac{u(r)^2}{2}
\quad\text{and}\quad
Q(r) = r^{\dSpace-1}\dot q(r)
\,.
\]
Then (omitting the dependency on $r$), 
\[
\dot q = u\cdot\dot u 
\quad\text{and}\quad
\ddot q = \dot u^2 + u\cdot\ddot u = \dot u^2 - \frac{\dSpace-1}{r} \dot q + u\cdot \nabla V(u)
\,,
\]
so that
\[
\dot Q = r^{\dSpace-1}\left(\ddot q + \frac{\dSpace-1}{r} \dot q \right) = r^{\dSpace-1}\bigl(\dot u^2+ u\cdot \nabla V(u)\bigr)
\,.
\]
According to the definition \cref{notation_vvvQuad_of_R} of $\vvvQuad{R}$, there exists a positive quantity $\delta$ (sufficiently small) so that, for every $w$ in $\rr^{\dState}$,
\begin{equation}
\label{lower_bound_w_cdot_nabla_V_of_w}
\abs{w}\ge R-\delta \implies w\cdot \nabla V(w)\ge \frac{w^2}{2}
\,.
\end{equation}
Let us proceed by contradiction and assume that $\sup_{r\in(0,+\infty)}\abs{u(r)}$ is not smaller than $R$. Since $u(\cdot)$ is stable at infinity and since the critical points of $V$ belong to the open ball $B_{\rr^{\dState}}(0,R-\delta)$, it follows that the set
\[
\bigl\{r\in[0,+\infty): \abs{u(r)}\ge R\bigr\} 
\]
is nonempty; let $\rOut$ denote the minimum of this set. For the same reason, the set
\[
\bigl\{r\in(\rOut,+\infty):\abs{u(r)} < R-\delta\bigr\}
\]
is also nonempty. Let $\rBack$ denote the infimum of this last set. It follows from these definitions that $\rBack$ is larger than $\rOut$ and that, for every $r$ in $(\rOut,\rBack)$, according to inequality \cref{lower_bound_w_cdot_nabla_V_of_w},
\begin{equation}
\label{dot_P_positive_between_r0_and_r1}
\dot Q(r) \ge r^{\dSpace-1}\left(\dot u^2(r) + \frac{u^2(r)}{2}\right) >0
\,.
\end{equation}
If on the one hand $\rOut$ equals $0$ then $\abs{u(0)}$ is not smaller than $R$ and, since $Q(0)$ equals $0$, it follows from inequality \cref{dot_P_positive_between_r0_and_r1} that $Q(\cdot)$ is positive on $(0,\rBack)$, so that the same is true for $\dot q(\cdot)$. Thus $q(\cdot)$ is strictly increasing on $[0,\rBack]$ and $\abs{u(\rBack)}$ must be larger than $\abs{u(\rOut)}$, a contradiction with the definition of $\rBack$. If on the other hand $\rOut$ is positive, then $\abs{u(\rOut)}$ is equal to $R$ and $\dot q(\rOut)$ is nonnegative so that the same is true for $Q(\rOut)$, and it again follows from inequality \cref{dot_P_positive_between_r0_and_r1} that $Q(\cdot)$ is positive on $(0,\rBack)$, yielding the same contradiction. Conclusion \cref{item:lem_a_priori_bound_on_profiles} of \cref{lem:global_flow_and_a_priori_bound_on_profiles} is proved. 
\end{proof}
\begin{notation}
For every positive quantity $R$ and every potential $V$ in $\vvvQuad{R}$, let 
\begin{equation}
\label{flow}
S_V:(0,+\infty)^2\times\rr^{2\dState}\to\rr^{2\dState}
\,,\quad
\bigl((\rInit,r),(\uInit,\vInit)\bigr)\mapsto S_V\bigl((\rInit,r),(\uInit,\vInit)\bigr)
\end{equation}
denote the (globally defined) flow of the (non-autonomous) differential system \cref{differential_system_radially_symmetric_stationary_order_1} for this potential $V$. In other words, for every $\rInit$ in $(0,+\infty)$ and $(\uInit,\vInit)$ in $\rr^{2\dState}$, the function 
\[
(0,+\infty)\to\rr^{2\dState}\,,\quad r\mapsto S_V\bigl((\rInit,r_1),(\uInit,\vInit)\bigr)
\]
is the solution of the differential system \cref{differential_system_radially_symmetric_stationary_order_1} for the initial condition $(\uInit,\vInit)$ at $r$ equals $\rInit$. According to \cref{subsec:differential_systems_gov_rad_symm_stat_sol}, the flow $S_V$ may be extended to the larger set
\[
(0,+\infty)^2\times\rr^{2\dState} \cup [0,+\infty)^2\times\rr^{\dState}\times\{0_{\rr^{\dState}}\}
\,;
\]
according to this extension, for every $u_0$ in $\rr^{\dState}$, the solution taking its values in the (one-dimensional) unstable manifold $\WuOriginPlus_V(u_0)$ reads:
\begin{equation}
\label{flow_unstable_manifold_origin}
[0,+\infty)\to\rr^{\dState}
\,,\quad
r\mapsto S_V\bigl((0,r),(u_0,0_{\rr^{\dState}})\bigr)
\,.
\end{equation}
\end{notation}
\section{Generic transversality among potentials that are quadratic past a given radius}
\label{sec:generic_transversality_potentials_quadratic_past_given_radius}
\subsection{Notation and statement}
Let us recall the notation $\ssss_V$ and $\ssss_{\VuInfty}$ introduced in \cref{notation_sss_V}.
\begin{proposition}
\label{prop:generic_transversality_potentials_quadratic_past_given_radius}
There exists a generic subset of $\vvvQuad{R}$ such that, for every potential $V$ in this subset, every radially symmetric stationary solution stable at infinity of the parabolic system \cref{parabolic_system_higher_space_dimension} (in other words, every $u$ in $\ssss_V$) is transverse. 
\end{proposition}
\subsection{Reduction to a local statement}
Let $V_1$ denote a potential function in $\vvvQuad{R}$ and $u_{1,\infty}$ denote a nondegenerate minimum point of $V_1$. According to the implicit function theorem, there exists a (small) neighbourhood $\nuRobust(V_1,u_{1,\infty})$ of $\vvvQuad{R}$ and a $\ccc^k$-function $u_\infty(\cdot)$ defined on $\nuRobust(V_1,u_{1,\infty})$ and with values in $\rr^{\dState}$ such that $u_\infty(V_1)$ equals $u_{1,\infty}$ and, for every $V$ in $\nuRobust(V_1,u_{1,\infty})$, $u_\infty(V)$ is a local minimum point of $V$. The following local generic transversality statement yields \cref{prop:generic_transversality_potentials_quadratic_past_given_radius} (as shown below). 
\begin{proposition}
\label{prop:local_generic_transversality_potentials_quadratic_past_given_radius}
There exists a neighbourhood $\nu_{\VOneuOneInfty}$ of $V_1$ in $\nuRobust(\VOneuOneInfty)$ and a generic subset $\nu_{\VOneuOneInftyGen}$ of $\nu_{\VOneuOneInfty}$ such that, for every $V$ in $\nu_{\VOneuOneInftyGen}$, every radially symmetric stationary solution stable close to $u_\infty(V)$ at infinity of the parabolic system \cref{parabolic_system_higher_space_dimension} (in other words, every $u$ in $\ssss_{\VuInftyOfV}$) is transverse.
\end{proposition}
\begin{proof}[Proof that \cref{prop:local_generic_transversality_potentials_quadratic_past_given_radius} yields \cref{prop:generic_transversality_potentials_quadratic_past_given_radius}]
Let us denote by $\vvvQuadMorse{R}$ the\break
dense open subset of $\vvvQuad{R}$ defined by the Morse property:
\begin{equation}
\label{def_vvv_Quad_R_Morse}
\vvvQuadMorse{R} = \left\{ V\in\vvvQuad{R} : \text{all critical points of $V$ are nondegenerate} \right\}
\,.
\end{equation}
Let $V_1$ denote a potential function in $\vvvQuadMorse{R}$. According to the Morse property its minimum points are isolated and since $V_1$ is in $\vvvQuad{R}$ they belong to the open ball $B_{\rr^d}(0,R)$, so that those minimum points are in finite number. Assume that \cref{prop:local_generic_transversality_potentials_quadratic_past_given_radius} holds. With the notation of this proposition, let us consider the following two intersections, at each time over all minimum points $u_{1,\infty}$ of $V_1$:
\begin{equation}
\nu_{V_1} = \bigcap\nu_{\VOneuOneInfty} 
\quad\text{and}\quad
\nu_{\VOneGen} = \bigcap\nu_{\VOneuOneInftyGen}
\,.
\end{equation}
Since those are finite intersections, $\nu_{V_1}$ is still a neighbourhood of $V_1$ in $\vvvQuad{R}$ and the set $\nu_{\VOneGen}$ is still a generic subset of $\nu_{V_1}$. This shows that the set 
\[
\{V\in\vvvQuadMorse{R} : \text{ every $u$ in $\ssss_{\VuInftyOfV}$ is transverse} \}
\] 
is locally generic. Applying \mainPaperLemLocalGenericityImpliesGlobalGenericity{} of \cite{JolyRisler_genericTransversalityTravStandFrontsPulses_2023} as in \mainPaperSectionReductionToLocalStatement{} of this reference shows that this local genericity implies the global genericity stated in \cref{prop:generic_transversality_potentials_quadratic_past_given_radius}, which is therefore proved. 
\end{proof}
\subsection{Proof of the local statement (\texorpdfstring{\cref{prop:local_generic_transversality_potentials_quadratic_past_given_radius}}{Proposition \ref{prop:local_generic_transversality_potentials_quadratic_past_given_radius}})}
\label{subsec:proof_local_gen}
\subsubsection{Setting}
\label{subsubsec:setting_proof_local_statement}
For the remaining part of this \namecref{sec:generic_transversality_potentials_quadratic_past_given_radius}, let us fix a potential function $V_1$ in $\vvvQuad{R}$ and a nondegenerate minimum point $u_{1,\infty}$ of $V_1$. Let $\nu$ be a neighbourhood of $V_1$ in $\vvvQuad{R}$, included in $\nuRobust(V_1,u_{1,\infty})$, and let $\varepsilon_1$ and $c_1$ be positive quantities, with $\nu$ and $\varepsilon_1$ and $c_1$ small enough so that the conclusions of \cref{prop:loc_centre_stab_manifold_at_infinity} hold. Let 
\[
\begin{aligned}
r_1 &= 1/c_1
&\ &\text{and}\ &
\mmm &= \rr^{\dState}\times B_{\rr^{\dState}}(u_{1,\infty},\varepsilon_1) 
&\ &\text{and}\ &
\Lambda &= \nu\,, \\
\text{and}\quad
\nnn &= (\rr^{2\dState})^2
&\ &\text{and}\ &
\www &= \{(A,B)\in\nnn : A = B\}
& & &
\,, 
\end{aligned}
\]
thus $\www$ is the diagonal of $\nnn$. Let $N$ denote an integer not smaller than $r_1$, and let us consider the functions
\[
\begin{aligned}
\Phiu : \rr^{\dState}\times\Lambda&\to\rr^{2\dState}
\,, &
(u_0,V)&\mapsto S_V\bigl((0,N),(u_0,0_{\rr^{\dState}})\bigr) \,, \\
\text{and}\quad 
\Phics : B_{\rr^{\dState}}(u_{1,\infty},\varepsilon_1)\times\Lambda&\to\rr^{2\dState} 
\,, &
(u_N,V)&\mapsto \bigl(u_N,\wcslocInftyPlus{V}(u_N,1/N)\bigr)
\,, 
\end{aligned}
\]
and the function
\begin{equation}
\label{def_Phi}
\Phi : \mmm\times\Lambda\to\nnn
\,, \quad
(m,V) = (u_0,u_N,V)\mapsto \bigl(\Phiu(u_0,V),\Phics(u_N,V)\bigr)
\,.
\end{equation}
\subsubsection{Equivalent characterizations of transversality}
Let us consider the set
\[
\ssss_{\Lambda,u_{1,\infty},N} = \bigl\{ (V,u): V\in\Lambda \text{ and } u \in\ssss_{\VuInftyOfV} \text{ and } u(N)\in B_{\rr^{\dState}}(u_{1,\infty},\varepsilon_1) \bigr\}
\,.
\]
\begin{proposition}
\label{prop:one_to_one_correspondence_with_diagonal_intersection}
The map 
\begin{equation}
\label{correspondence_between_Phi_minus_1_of_W_and_ssss_Lambda_N}
\Phi^{-1}(\www)\to\ssss_{\Lambda,u_{1,\infty},N}
\,,\quad
(u_0,u,V)\mapsto \Bigl(V,r\mapsto S_V\bigl((0,r),(u_0,0_{\rr^{\dState}}\bigr)\Bigr)
\end{equation}
is well defined and one-to-one.
\end{proposition}
\begin{proof}
The image by $\Phi$ of a point $(u_0,u_N,V)$ of $\mmm\times\Lambda$ belongs to the diagonal $\www$ of $\nnn$ if and only if $\Phiu(u_0,V)$ equals $\Phics(u_N,V)$, and in this case the function $u:r\mapsto S_V\bigl((0,r),(u_0,0_{\rr^{\dState}}\bigr)$ belongs to $\ssss_{\VuInftyOfV}$ and $u(N)$ (which is equal to $u_N$) belongs to $B_{\rr^{\dState}}(u_{1,\infty},\varepsilon_1)$, so that $(V,u)$ belongs to $\ssss_{\Lambda,u_{1,\infty},N}$. The map \cref{correspondence_between_Phi_minus_1_of_W_and_ssss_Lambda_N} above is thus well defined.

Now, for every $(V,u)$ in $\ssss_{\Lambda,u_{1,\infty},N}$, if we denote by $u_0$ the limit $\lim_{r\to0^+}u(r)$ and by $u_N$ the vector $u(N)$, then $(u_0,u_N,V)$ is the only possible antecedent of $(V,u)$ by the map \cref{correspondence_between_Phi_minus_1_of_W_and_ssss_Lambda_N}. In addition, 
\[
S_V\bigl((0,N),(u_0,0_{\rr^{\dState}})\bigr) = \bigl(u_N,\dot u(N)\bigr)
\,,
\]
and since $u(r)$ goes to $u_\infty(V)$ as $r$ goes to $+\infty$, the vector $\bigl(u(N),\dot u(N),1/N\bigr)$ must belong to the centre-stable manifold $\WcsInftyPlus_V\bigl(u_\infty(V)\bigr)$ of $u_\infty(V)$, so that, according to the definition of $\wcslocInftyPlus{V}$, 
\[
\dot u(N) = \wcslocInftyPlus{V}\bigl(u(N),1/N\bigr)
\,,
\]
and this yields the equality between $\Phiu(u_0,V)$ and $\Phics(u_N,V)$. Thus $\Phi(V,u)$ belongs to $\www$ and $(u_0,u_N,V)$ belongs to $\Phi^{-1}(\www)$. \Cref{prop:one_to_one_correspondence_with_diagonal_intersection} is proved. 
\end{proof}
\begin{proposition}
\label{prop:equivalence_transversality}
For every potential function $V$ in $\Lambda$, the following two statements are equivalent.
\begin{enumerate}
\item The image of the function $\mmm\to\nnn$, $m\mapsto\Phi(m,V)$ is transverse to $\www$.
\label{item:lem_equivalence_transversality_Phi}
\item Every $u$ in $\ssss_{\VuInftyOfV}$ such that $u(N)$ is in $B_{\rr^{\dState}}(u_{1,\infty},\varepsilon_1)$ is transverse. 
\label{item:lem_equivalence_transversality_fronts}
\end{enumerate}
\end{proposition}
\begin{remark}
According to \cref{prop:transversality_of_homogeneous_solutions}, for every $V$ in $\Lambda$, the constant function $r\mapsto u_\infty(V)$, which belongs to $\ssss_V$, is already (a priori) known to be transverse, therefore only nonconstant solutions matter in statement \cref{item:lem_equivalence_transversality_fronts} of this proposition. 
\end{remark}
\begin{proof}
Let us consider $(m_2,V_2)$ in $\mmm\times\Lambda$ such that $\Phi(m_2,V_2)$ is in $\www$, let $(u_{2,0},u_{2,N})$ denote the components of $m_2$, and let $r\mapsto u_2(r)$ and $r\mapsto U_2(r)$ denote the functions satisfying, for all $r$ in $[0,+\infty)$, 
\[
U_2(r) = \bigl(u_2(r),\dot u_2(r)\bigr) = S_V\bigl((0,r),(u_{2,0},0_{\rr^{\dState}}\bigr)
\,.
\]
Let us consider the map
\[
\Delta\Phi : \mmm\to\rr^{2\dState}\,, \quad
(u_0,u_N)\mapsto \Phiu(u_0,V_2) - \Phics(u_N,V_2)
\,,
\]
and let us write, only for this proof, $D\Phi$ and $D\Phiu$ and $D\Phics$ and $D(\Delta\Phi)$ for the differentials of $\Phi$ and $\Phiu$ and $\Phics$ and $\Delta\Phi$ at $(m_2,V_2)$ and with respect to all variables in $\mmm$ (but not with respect to $V$). According to \cref{def:transversality_rad_symm_solution_stable_at_infinity}, the transversality of $u_2$ is defined as the transversality of the intersection $\WuOriginPlus_{V_2}\cap \iota^{-1}\Bigl(\WcsInftyPlus_{V_2}\bigl(u_{\infty}(V_2)\bigr)\Bigr)$ along the trajectory of $U_2$. This transversality can be considered at a single point, no matter which, of the trajectory $U_2\bigl((0,+\infty)\bigr)$, in particular at the point $\Phiu(u_{2,0},V_2)$ which is equal to $\Phics\bigl(u_2(N),V2\bigr)$, and is equivalent to the transversality of the $\dState$-dimensional manifolds 
\[
\WuOriginPlus_{V_2}\cap\bigl(\rr^{2\dState}\times\{N\}\bigr)
\quad\text{and}\quad
\iota^{-1}\Bigl(\WcsInftyPlus_{V_2}\bigl(u_{\infty}(V_2)\bigr)\Bigr)\cap\bigl(\rr^{2\dState}\times\{N\}\bigr)
\]
in $\rr^{2\dState}\times\{N\}$. It is therefore equivalent to the surjectivity of the map $D(\Delta\Phi)$ (statement \cref{item:transversality_Delta_Phi} in \cref{lem:equivalence_transversality_linear_computation} below). On the other hand, the image of the function $\mmm\to\nnn$, $m\mapsto\Phi(m,V_2)$ is transverse at $\Phi(m,V_2)$ to the diagonal $\www$ of $\nnn$ if and only if the image of $D\Phi$ contains a complementary space of this diagonal (statement \cref{item:transversality_Phi} in \cref{lem:equivalence_transversality_linear_computation} below)). Thus \cref{prop:equivalence_transversality} is a consequence of the next lemma. 
\begin{lemma}
\label{lem:equivalence_transversality_linear_computation}
The following two statements are equivalent.
\begin{enumerate}[(A)]
\item The image of $D\Phi$ contains a complementary subspace of the diagonal $\www$ of $\nnn$. 
\label{item:transversality_Phi}
\item The map  $D(\Delta\Phi)$ is surjective. 
\label{item:transversality_Delta_Phi}
\end{enumerate}
\end{lemma}
\begin{proof}
If statement \cref{item:transversality_Phi} holds, then, for every $(\alpha,\beta)$ in $\nnn$, there exist $\gamma$ in $\rr^{2\dState}$ and $\delta m$ in $T_{m_2}\mmm$ such that
\begin{equation}
\label{complementary_of_diagonal}
(\gamma,\gamma) + D\Phi\cdot\delta m = (\alpha,\beta)
\,,
\end{equation}
so that
\begin{equation}
\label{DDelta_Phi_surjective}
D(\Delta\Phi)\cdot\delta m = \alpha-\beta
\,,
\end{equation}
and statement \cref{item:transversality_Delta_Phi} holds. Conversely, if statement \cref{item:transversality_Delta_Phi} holds, then, for every $(\alpha,\beta)$ in $\nnn$, there exists $\delta m$ in $T_{m_2}\mmm$ such that \cref{DDelta_Phi_surjective} holds, and as a consequence, if $(\delta u_0,\delta u_N)$ denote the components of $\delta m$, then $\alpha - D\Phiu(\delta u_0)$ is equal to $\beta - D\Phics(\delta u_N)$, and if this vector is denoted by $\gamma$, then equality \cref{complementary_of_diagonal} holds, and this shows that statement \cref{item:transversality_Phi} holds. 
\end{proof}
As explained above, \cref{prop:equivalence_transversality} follows from \cref{lem:equivalence_transversality_linear_computation}, and is therefore proved. 
\end{proof}
\subsubsection{Checking hypothesis \texorpdfstring{\mainPaperItemDimensionsSardSmaleTheorem{} of \mainPaperSardSmaleTheorem{} of \cite{JolyRisler_genericTransversalityTravStandFrontsPulses_2023}}{1 of Theorem 4.2 of JolyRislerGenericTransversalityTravStandFrontsPulses2021}}
\label{subsubsec:checking_hypothesis_dimensions}
The function $\Phi$ is as regular as the flow $S_V$, thus of class $\ccc^k$. It follows from the definitions of $\mmm$ and $\nnn$ and $\www$ that
\[
\dim(\mmm) - \codim(\www) = (\dState + \dState) - 2\dState = 0
\,,
\]
so that hypothesis \mainPaperItemDimensionsSardSmaleTheorem{} of \mainPaperSardSmaleTheorem{} of \cite{JolyRisler_genericTransversalityTravStandFrontsPulses_2023} is fulfilled. 
\subsubsection{Checking hypothesis \texorpdfstring{\mainPaperItemTransversalitySardSmaleTheorem{} of \mainPaperSardSmaleTheorem{} of \cite{JolyRisler_genericTransversalityTravStandFrontsPulses_2023}}{2 of Theorem 4.2 of JolyRislerGenericTransversalityTravStandFrontsPulses2021}}
For every $V$ in $\vvvQuad{R}$, let us recall the notation $S_V$ introduced in \cref{flow,flow_unstable_manifold_origin} for the flow of the differential system \cref{differential_system_radially_symmetric_stationary_order_1}.
Take $(m_2,V_2)$ in the set $\Phi^{-1}(\www)$. Let $(u_{2,0},u_{2,N})$ denote the components of $m_2$, and, for every $r$ in $(0,+\infty)$, let us write
\[
U_2(r) = \bigl(u_2(r),v_2(r)\bigr) = S_{V_2}\bigl((0,r),(u_{2,0},0_{\rr^{\dState}})\bigr)
\,.
\]
Let us write
\[
D\Phi
\quad\text{and}\quad
D\Phiu
\quad\text{and}\quad
D\Phics
\]
for the \emph{full} differentials (with respect to arguments $m$ in $\mmm$ \emph{and} $V$ in $\Lambda$) of the three functions $\Phi$ and $\Phiu$ and $\Phics$ respectively at the points $\bigl(u_{2,0},u_{2,N},V_2\bigr)$, 
$\bigl(u_{2,0},V_2\bigr)$ and $\bigl(u_{2,N},V_2\bigr)$. Checking hypothesis \mainPaperItemTransversalitySardSmaleTheorem{} of \mainPaperSardSmaleTheorem{} of \cite{JolyRisler_genericTransversalityTravStandFrontsPulses_2023} amounts to prove that 
\begin{equation}
\label{transversality_of_DPhi}
\img(D\Phi) + \www = \nnn
\,.
\end{equation}
If $u_2(\cdot)$ is constant (that is, identically equal to $u_\infty(V_2)$), then equality \cref{transversality_of_DPhi} follows from \cref{prop:transversality_of_homogeneous_solutions}. Thus, let us assume that $u_2(\cdot)$ is nonconstant. In this case, equality \cref{transversality_of_DPhi} is a consequence of the following lemma.
\begin{lemma}
\label{lem:existence_suitable_perturbation_direction_W}
For every nonzero vector $(\phi_2,\psi_2)$ in $\rr^{2\dState}$, there exists a function $W$ in $\CkbFull$ such that
\begin{align}
\label{supp_W_included_in_BR}
&\supp(W)\subset B_{\rr^d}(0,R)\,,\\
\text{and}\quad
\label{transversality_Phi_condition_dPhiu}
&\bigl\langle D\Phiu\cdot(0,0,W)\bigm|(\phi_2,\psi_2)\bigr\rangle\not=0
\,,\\
\text{and}\quad
\label{transversality_Phi_condition_dPhis}
&D\Phics\cdot(0,0,W) = 0_{\rr^{2\dState}}
\,.
\end{align}
\end{lemma}
\begin{proof}[Proof that \cref{lem:existence_suitable_perturbation_direction_W} yields equality \cref{transversality_of_DPhi}]
Inequality \cref{transversality_Phi_condition_dPhiu} shows that the orthogonal complement, in $\rr^{2\dState}$, of the directions that can be reached by $D\Phiu\cdot(0,0,W)$ for potentials $W$ satisfying \cref{supp_W_included_in_BR,transversality_Phi_condition_dPhis} is reduced to $0_{\rr^{2\dState}}$; in other words, all directions of $\rr^{2\dState}$ can be reached by that means. This shows that
\[
\img(D\Phi) \supset \rr^{2\dState}\times\{0_{\rr^{2\dState}}\}
\,,
\]
and since the subspace at the right-hand side of this inclusion is transverse to $\www$ in $\rr^{4\dState}$, this proves equality \cref{transversality_of_DPhi} (and shows that hypothesis \mainPaperItemTransversalitySardSmaleTheorem{} of \mainPaperSardSmaleTheorem{} of \cite{JolyRisler_genericTransversalityTravStandFrontsPulses_2023} is fulfilled).
\end{proof}
\begin{proof}[Proof of \cref{lem:existence_suitable_perturbation_direction_W}]
Let $(\phi_2,\psi_2)$ denote a nonzero vector in $\rr^{2\dState}$, let $W$ be a function in $\CkbFull$ satisfying the inclusion 
\begin{equation}
\label{condition_supp_W_proof_lemma_existence_W}
\supp(W)\subset B_{\rr^d}(0,R) \setminus \widebar{B}_{\rr^{\dState}}(u_{1,\infty},\varepsilon_1)
\,,
\end{equation}
and observe that inclusion \cref{supp_W_included_in_BR} and equality \cref{transversality_Phi_condition_dPhis} follow from this inclusion \cref{condition_supp_W_proof_lemma_existence_W}. Let us consider the linearization of the differential system \cref{differential_system_radially_symmetric_stationary_order_2}, for the potential $V_2$, around the solution $r\mapsto U_2(r)$:
\begin{equation}
\label{linearized_differential_system}
\frac{d}{dr}
\begin{pmatrix}
\delta u(r) \\ \delta v(r) 
\end{pmatrix}
= \begin{pmatrix}
0 & \id \\ D^2 V_2 \bigl(u_2(r)\bigr) & -\frac{\dSpace-1}{r}
\end{pmatrix}
\begin{pmatrix}
\delta u(r) \\ \delta v(r) 
\end{pmatrix}
\,,
\end{equation}
and let $T(r,r')$ denote the family of evolution operators obtained by integrating this linearized differential system between $r$ and $r'$. It follows from the variation of constants formula that 
\begin{equation}
\label{expression_DPhi_trav_front}
D\Phiu\cdot(0,0,W) = \int_{-\infty}^{N} T(r,N) \Bigl(0,\nabla W\bigl(u_2(r)\bigr)\Bigr)\, dr
\,.
\end{equation}
For every $r$ in $(0,+\infty)$, let $T^*(r,N)$ denote the adjoint operator of $T(r,N)$, and let 
\begin{equation}
\label{def_phi_psi}
\bigl(\phi(r),\psi(r)\bigr) = T^*(r,N)\cdot(\phi_2,\psi_2)
\,.
\end{equation}
According to expression \cref{expression_DPhi_trav_front}, inequality \cref{transversality_Phi_condition_dPhiu} reads
\[
\int_{-\infty}^{N}\Bigl\langle\Bigl(0,\nabla W\bigl(u_2(r)\bigr)\Bigr) \,\Bigm|\, T^*(r,N)\cdot(\phi_2,\psi_2) \Bigr\rangle\, dr \not= 0
\,,
\]
or equivalently
\begin{equation}
\label{reach_all_directions_condition_with_psi}
\int_{-\infty}^{N}\nabla W\bigl(u_2(r)\bigr)\cdot \psi(r) \, dr \not= 0
\,.
\end{equation}
Due to the expression of the linearized differential system \cref{linearized_differential_system}, $(\phi,\psi)$ is a solution of the adjoint linearized system
\begin{equation}
\label{adjoint_linearized_system}
\begin{pmatrix}
\dot\phi(r) \\ \dot\psi(r) 
\end{pmatrix}
= - \begin{pmatrix}
0 & D^2 V_2 \bigl(u_2(r)\bigr) \\
\id & -\frac{\dSpace-1}{r}
\end{pmatrix}  
\begin{pmatrix}
\phi(r) \\ \psi(r) 
\end{pmatrix}
\,.
\end{equation}
According to \cref{lem:dot_u_does_not_vanish_and_values_reached_once} (and since $u_2(\cdot)$ was assumed to be \emph{nonconstant}), there exists positive quantity $\rOnce$ such that, if we denote by $\IOnce$ the interval $(0,\rOnce]$, then $\dot u_2(\cdot)$ \emph{does not vanish} on $\IOnce$, and, for all $r^*$ in $\IOnce$ and $r$ in $\rr$, 
\begin{equation}
\label{values_reached_only_once_trav}
u_2(r) = u_2 (r^*) \implies r = r^*
\,.
\end{equation}
In addition, up to replacing $\rOnce$ by a smaller positive quantity, it may be assumed that the following conclusions hold: 
\[
u_2(\IOnce)\cap \widebar{B}_{\rr^{\dState}}(u_{1,\infty},\varepsilon_1) = \emptyset
\,.
\]
To complete the proof three cases have to be considered. 
\paragraph*{Case 1.} There exists $r^*$ in $\IOnce$ such that $\psi(r^*)$ is \emph{not} collinear to $\dot u_2(r^*)$. 

In this case, the construction of a potential function $W$ satisfying inclusion \cref{condition_supp_W_proof_lemma_existence_W} and inequality \cref{transversality_Phi_condition_dPhiu} (and thus the conclusions of \cref{lem:existence_suitable_perturbation_direction_W}) is the same as in the proof of \mainPaperLemPerturbationPotReachingAGivenDirection{} of \cite{JolyRisler_genericTransversalityTravStandFrontsPulses_2023}. 

If case 1 does \emph{not} occur, then $\psi(r)$ is collinear to $\dot u_2(r)$, and since $\dot u_2(\cdot)$ does not vanish on $\IOnce$, there exists a $\ccc^1$-function $\alpha:\IOnce \to\rr$ such that, for every $r$ in $\IOnce$,
\begin{equation}
\label{psi_equals_alpha_times_derivative_of_uzero}
\psi(r) = \alpha(r) \dot u_2(r)
\,.
\end{equation} 
The next cases 2 and 3 differ according to whether the function $\alpha(\cdot)$ is constant or not. 

\paragraph*{Case 2.} For every $r$ in $\IOnce$, equality \cref{psi_equals_alpha_times_derivative_of_uzero} holds for some \emph{nonconstant} function $\alpha(\cdot)$.

In this case there exists $r^*$ in $\IOnce$ such that $\dot\alpha(r^*)$ is nonzero, and again the construction of a potential function $W$ satisfying inclusion \cref{condition_supp_W_proof_lemma_existence_W} and inequality \cref{transversality_Phi_condition_dPhiu} (and thus the conclusions of \cref{lem:existence_suitable_perturbation_direction_W}) is the same as in the proof of \mainPaperLemPerturbationPotReachingAGivenDirection{} of \cite{JolyRisler_genericTransversalityTravStandFrontsPulses_2023}. 

\paragraph*{Case 3.} For every $r$ in $\IOnce$, $\psi(r)=\alpha \dot u_2(r)$ for some real (constant) quantity $\alpha$. 

In this case the quantity $\alpha$ cannot be $0$ or else, due to \cref{adjoint_linearized_system,psi_equals_alpha_times_derivative_of_uzero}, both $\phi(\cdot)$ and $\psi(\cdot)$ would identically vanish on $\IOnce$ and thus on $(0,+\infty)$, a contradiction with the assumptions of \cref{lem:existence_suitable_perturbation_direction_W}. Thus, without loss of generality, we may assume that $\alpha$ is equal to $1$. If $\supp(W)$ is included in a sufficiently small neighbourhood of $u_{2,0}$, then $W(\cdot)$ vanishes on $u_2\bigl([\rOnce,N]\bigr)$ and the integral on the left-hand side of inequality \cref{reach_all_directions_condition_with_psi} reads
\[
\int_0^{\rOnce}\nabla W\bigl(u_2(r)\bigr)\cdot \dot u_2(r) \, dr = W\bigl(u_2(\rOnce)\bigr) - W(u_{2,0}) = - W(u_{2,0})
\,,
\]
so that inequality \cref{reach_all_directions_condition_with_psi} holds as soon as $W(u_{2,0})$ is nonzero. \Cref{lem:existence_suitable_perturbation_direction_W} is proved. 
\end{proof}
\begin{remark}
By contrast with the proof of the generic elementarity of standing pulses in \cite{JolyRisler_genericTransversalityTravStandFrontsPulses_2023}, case 3 above cannot be easily precluded. Indeed, let us assume that, for every $r$ in $\IOnce$, $\psi(r)$ is equal to $\alpha\dot u_2(r)$ for some nonzero (constant) quantity $\alpha$. Without loss of generality, we may assume that $\alpha$ is equal to $1$. Then, it follows from the second equation of \cref{adjoint_linearized_system} that, still for every $r$ in $\IOnce$ (omitting the dependency on $r$),
\[
\phi = \frac{\dSpace-1}{r}\psi - \dot \psi = \frac{\dSpace-1}{r}\dot u_2 - \ddot u_2 = \frac{2(\dSpace-1)}{r}\dot u_2 - \nabla V_2(u_2)
\,,
\]
and it follows from the first equation of \cref{adjoint_linearized_system} that
\[
-D^2V_2(u_2)\dot u_2 = \dot \phi = - \frac{2(\dSpace-1)}{r^2}\dot u_2 + \frac{2(\dSpace-1)}{r}\ddot u_2 - D^2 V_2(u_2)\dot u_2
\,,
\]
and thus, after simplification,
\[
\ddot u_2 = \frac{1}{r}\dot u_2
\,,\quad\text{or equivalently}\quad
\dot u_2 = \frac{r}{\dSpace}\nabla V(u_2)
\,.
\]
As illustrated by equality \cref{estimate_dot_u_close_to_origin}, this last equality indeed holds if $\nabla V_2$ is constant on the set $u_2(\IOnce)$. Case 3 can therefore not be a priori precluded, and if it may be argued that this case is ``unlikely'' (non generic), the direct argument provided above in this case is simpler. By contrast, in \cite{JolyRisler_genericTransversalityTravStandFrontsPulses_2023} for standing pulses in space dimension one ($\dSpace$ equal to $1$), this case could not occur because $\psi$ was assumed to be nonzero on the symmetry subspace, defined here as $\{(v,r)=(0_{\rr^{\dState}},0)\}$, see \cref{notation_sssSym}. 
\end{remark}
\subsubsection{Conclusion}
As seen in \cref{subsubsec:checking_hypothesis_dimensions}, hypothesis 1 of \mainPaperSardSmaleTheorem{} of \cite{JolyRisler_genericTransversalityTravStandFrontsPulses_2023} is fulfilled for the function $\Phi$ defined in \cref{def_Phi}, and since \cref{lem:existence_suitable_perturbation_direction_W} yields equality \cref{transversality_of_DPhi}, hypothesis 2 of this theorem is also fulfilled. The conclusion of this theorem ensures that there exists a generic subset $\LambdaGenOf{N}$ of $\Lambda$ such that, for every $V$ in $\LambdaGenOf{N}$, the image of the function $\mmm\to\nnn$, $m\mapsto \Phi(m,V)$ is transverse to the diagonal $\www$ of $\nnn$. According to \cref{prop:equivalence_transversality}, it follows that every $u$ in $\ssss_{\VuInftyOfV}$ such that $u(N)$ is in $B_{\rr^{\dState}}(u_{1,\infty},\varepsilon_1)$ is transverse. The set
\[
\LambdaGen = \bigcap_{N\in\nn,\,N\ge r_0}\LambdaGenOf{N}
\]
is still a generic subset of $\Lambda$. For every $V$ in $\LambdaGen$ and every $u$ in $\ssss_{\VuInftyOfV}$, since $u(r)$ goes to $u_\infty(V)$ as $r$ goes to $+\infty$, there exists $N$ such that $u(N)$ is in $B_{\rr^{\dState}}(u_{1,\infty},\varepsilon_1)$, and according to the previous statements $u$ is transverse. In other words, the conclusions of \cref{prop:local_generic_transversality_potentials_quadratic_past_given_radius} hold with: 
\[
\nu_{\VOneuOneInfty} = \nu = \Lambda
\quad\text{and}\quad
\nu_{\VOneuOneInftyGen} = \LambdaGen
\,.
\]
\section{Proof of the main results}
\label{sec:proof_main_results}
\Cref{prop:generic_transversality_potentials_quadratic_past_given_radius} shows the genericity of the property considered in \cref{thm:main}, but only inside the space $\vvvQuad{R}$ of the potentials that are quadratic past some radius $R$. In this section, the arguments will be adapted to obtain the genericity of the same property in the space $\vvvFull$ (that is $\CkFull$) of all potentials, endowed with the extended topology (see \cref{subsec:space_of_potentials}). They are identical to those of \mainPaperSectionProofMainResults{} of \cite{JolyRisler_genericTransversalityTravStandFrontsPulses_2023}. Let us recall the notation $\ssss_V$ introduced in \cref{notation_sss_V}, and, for every positive quantity $R$, let us consider the set
\[
\ssss_{V,R} = \left\{u\in\ssss_V:\sup_{r\in[0,+\infty)}\abs{u(r)}\le R\right\}
\,.
\]
Exactly as shown in \mainPaperSubsectionProofConclusionTravFronts{} of \cite{JolyRisler_genericTransversalityTravStandFrontsPulses_2023}, \cref{thm:main} follows from the next proposition. 
\begin{proposition}
\label{prop:genericity_transv_tf_up_to_R}
For every positive quantity $R$, there exists a generic subset $\vvvFullTransvSsss{R}$ of $\vvvFull$ such that, for every potential $V$ in this subset, every radially symmetric stationary solution stable at infinity in $\ssss_{V,R}$ is transverse. 
\end{proposition}
\begin{proof}
Let $R$ denote a positive quantity, let $V_1$ denote a potential function in $\vvvQuad{(R+1)}$, and let $u_{1,\infty}$ denote a nondegenerate minimum point of $V_1$. Let us consider the neighbourhood $\nu_{\VOneuOneInfty}$ of $V_1$ in $\vvvQuad{(R+1)}$ provided by \cref{prop:local_generic_transversality_potentials_quadratic_past_given_radius} for these objects, together with the quantities $\varepsilon_1$, $c_1$, and $r_1$ introduced in \cref{subsubsec:setting_proof_local_statement}. Up to replacing $\nu_{\VOneuOneInfty}$ by its interior, we may assume that it is open in $\vvvQuad{(R+1)}$. As in \cref{subsubsec:setting_proof_local_statement}, let us consider an integer $N$ not smaller than $r_1$, and the same function $\Phi:\mmm\times\Lambda\to\nnn$ as in \cref{def_Phi}. 

Here is the sole difference with the setting of \cref{subsubsec:setting_proof_local_statement}: by contrast with the non-compact set $\mmm$ defining the departure set of $\Phi$, let us consider the compact subset $\mmm_N$ defined as:
\[
\mmm_N = B_{\rr^{\dState}}(0_{\rr^{\dState}},N) \times B_{\rr^{\dState}}(u_{1,\infty},\varepsilon_1)
\,.
\]
Thus the integer $N$ now serves two purposes: the ``time'' (radius) at which the intersection between unstable and centre-stable manifolds is considered, and the radius of the ball containing the departure points of the unstable manifolds that are considered. These purposes are independent (two different integers instead of the single integer $N$ may as well be introduced). Let us consider the set:
\[
\ooo_{V_1,u_{1,\infty},N} = \left\{V\in\nu_{\VOneuOneInfty}:\Phi(\mmm_N,V)\text{ is transverse to $\www$ in $\nnn$}
\right\}
\,.
\]
As shown in \cref{prop:equivalence_transversality}, this set $\ooo_{V_1,u_{1,\infty},N}$ is made of the potential functions $V$ in $\nu_{\VOneuOneInfty}$ such that every $u$ in $\ssss_{\VuInftyOfV}$ such that $u(N)$ is in $B_{\rr^{\dState}}(u_{1,\infty},\varepsilon_1)$ \emph{and} $u(0)$ is in $B_{\rr^{\dState}}(0_{\rr^{\dState}},N)$, is transverse. This set contains the generic subset $\nu_{\VOneuOneInftyGen} = \LambdaGen$ of $\nu_{\VOneuOneInfty}$ and is therefore generic (thus, in particular, dense) in $\nu_{\VOneuOneInfty}$. By comparison with $\nu_{\VOneuOneInftyGen}$, the additional feature of this set $\ooo_{V_1,u_{1,\infty},N}$ is that it is \emph{open}: exactly as in the proof of \mainPaperLemDensityOpennessOooExtension{} of \cite{JolyRisler_genericTransversalityTravStandFrontsPulses_2023}, this openness follows from the intrinsic openness of a transversality property and the compactness of $\mmm_N$. 

Let us make the additional assumption that the potential $V_1$ is a Morse function. Then, the set of minimum points of $V_1$ is finite and depends smoothly on $V$ in a neighbourhood $\nuRobust(V_1)$ of $V_1$. Intersecting the sets $\nu_{\VOneuOneInfty}$ and $\ooo_{V_1,u_{1,\infty},N}$ above over all the minimum points $u_{1,\infty}$ of $V_1$ provides an open neighbourhood $\nu_{V_1}$ of $V_1$ and an open dense subset $\ooo_{V_1,N}$ of $\nu_{V_1}$ such that, for all $V$ in $\nu_{V_1}$, every radially symmetric stationary solution stable close to a minimum point of $V$ at infinity, and equal at origin to some point of $B_{\rr^{\dState}}(0_{\rr^{\dState}},N)$, is transverse. 

Denoting by $\interior(A)$ the interior of a set $A$ and using the notation of \mainPaperSubsectionTopologicalPropertiesRestrictionMaps{} of \cite{JolyRisler_genericTransversalityTravStandFrontsPulses_2023}, let us introduce the sets
\[
\begin{aligned}
\tilde{\nu}_{V_1} &= \res_{R,\infty}^{-1}\circ\res_{R,(R+1)}(\nu_{V_1}) \,, \\
\text{and}\quad
\tilde{\ooo}_{V_1,N} &= \res_{R,\infty}^{-1}\circ\res_{R,(R+1)}(\ooo_{V_1,N}) \,, \\
\text{and}\quad
\tilde{\ooo}^{\ext}_{V_1,N} &= \tilde{\ooo}_{V_1,N} \sqcup \interior\bigl(\vvvFull\setminus\tilde{\nu}_{V_1}\bigr)
\,.
\end{aligned}
\]
It follows from these definitions that $\tilde{\ooo}^{\ext}_{V_1,N}$ is a dense open subset of $\vvvFull$ (for more details, see \mainPaperLemDensityOpennessTildeOooExtension{} of \cite{JolyRisler_genericTransversalityTravStandFrontsPulses_2023}). 

Since $\vvvQuad{(R+1)}$ is a separable space, it is second-countable, and can be covered by a countable number of sets of the form ${\nu}_{V_1}$. With symbols, there exists a countable family $(V_{1,i})_{i\in\nn}$ of potentials of $\vvvQuadMorse{(R+1)}$ so that 
\[
\vvvQuadMorse{(R+1)} = \bigcup_{i\in\nn}{\nu}_{V_{1,i}}
\,.
\]
Let us consider the set 
\[
\vvvFullTransvSsss{R} = \vvvFullMorse\cap\left(\bigcap_{(i,N)\in\nn^2}\tilde{\ooo}^{\ext}_{V_{1,i},N}\right)
\,,
\]
where $\vvvFullMorse$ is the set of potentials in $\vvvFull$ which are Morse functions. This set is a countable intersection of dense open subsets of $\vvvFull$, and is therefore a generic subset of $\vvvFull$. And, for every potential $V$ in this set $\vvvFullTransvSsss{R}$, every radially symmetric stationary solution stable at infinity in $\ssss_{V,R}$ is transverse (for more details, see \mainPaperLemCheckTransversalityExtension{} of \cite{JolyRisler_genericTransversalityTravStandFrontsPulses_2023}). \Cref{prop:genericity_transv_tf_up_to_R} is proved. 
\end{proof}
As already mentioned at the beginning of this \namecref{sec:proof_main_results}, \cref{thm:main} follows from \cref{prop:genericity_transv_tf_up_to_R}. Finally, \cref{cor:insight_main_result} follows from \cref{thm:main} (for more details, see \mainPaperSubsectionProofConclusionsOneToFourCorMain{} of \cite{JolyRisler_genericTransversalityTravStandFrontsPulses_2023}).
\paragraph*{Acknowledgements}
This paper owes a lot to numerous fruitful discussions with Romain Joly, about both its content and the content of the companion paper \cite{JolyRisler_genericTransversalityTravStandFrontsPulses_2023} written in collaboration with him. 
\printbibliography
\bigskip
\mySignature
%
%
\end{document}